\documentclass[12pt]{amsart}
\usepackage[english]{babel}
\usepackage{amsmath}
\usepackage{amsthm}
\usepackage{anysize}
\usepackage{mathtools}
\usepackage{mathrsfs}
\usepackage{amssymb}
\usepackage{xcolor}
\usepackage{xfrac}
\usepackage{esint}
\usepackage{graphicx}
\usepackage{float}
\usepackage{array}
\usepackage{enumitem}
\usepackage{tikz-cd}
\usepackage{centernot}
\usepackage{stmaryrd}
\usepackage{comment}
\excludecomment{confidential}
\usepackage{graphicx,color}
\usepackage{tikz,pgfplots,pgf}
\usepackage[font=small]{caption}

\newtheorem{theorem}{Theorem}[section]
\newtheorem{lemma}{Lemma}[section]
\newtheorem{proposition}{Proposition}[section]
\newtheorem{corollary}{Corollary}[section]
\theoremstyle{definition}

\theoremstyle{definition}

\newtheorem{definition}{Definition}[section]

\usepackage[english]{babel}
\usepackage{amsmath}
\usepackage{amsthm}
\usepackage{mathtools}
\usepackage{mathrsfs}
\usepackage{amssymb}
\usepackage{xcolor}
\usepackage{xfrac}
\usepackage{esint}
\usepackage{graphicx}
\usepackage{float}
\usepackage{array}
\usepackage{enumitem}
\usepackage[new]{old-arrows}
\usepackage[all,cmtip]{xy}
\usepackage{tikz-cd}
\usepackage{centernot}
\usepackage{stmaryrd}
\usepackage{comment}
\usepackage{bbm}

\newcommand{\mc}{\mathcal}

\newcommand{\R}{\mathbb{R}}
\newcommand{\N}{\mathbb{N}}

\renewcommand{\O}{\Omega}

\renewcommand{\a}{\alpha}

 \renewcommand{\o}{\omega}

\DeclareMathOperator{\essran}{essran}
\DeclareMathOperator{\supp}{supp}

\DeclareMathOperator{\dist}{dist}
\DeclareMathOperator{\pr}{pr}

\numberwithin{equation}{section}

\makeatletter
\@namedef{subjclassname@2020}{\textup{2020} Mathematics Subject Classification}
\makeatother

\begin{document}

\title[Distributional embeddings of the first limit BRS space]{Distributional embeddings of the first limit Bourgain--Rosenthal--Schechtman space}

\author[J. C. Sampedro]{Juan Carlos Sampedro} \thanks{The author has been supported by the Ministry of Science and Innovation of Spain under the
	Research Grant PID2024–155890NB-I00}
\address{Departamento de Matemática Aplicada y Ciencias de la Computación \\
	Avenida de los Castros 46 \\
	Universidad de Cantabria (UC) \\
	Santander, 39005, Spain.}
\email{juancarlos.sampedro@unican.es}



	\begin{abstract}
			We classify the distributional self-embeddings of the centered first limit Bourgain--Rosenthal--Schechtman space \(R_\omega^{p,0}\), \(1<p<\infty\). Using a Boolean rigidity principle for its canonical independent-sum realization, we show that every such embedding is induced by a finite packing of Bernoulli factors. As a consequence, we also prove that \(R_\omega^{p,0}\) admits no proper non-zero internal compressions. Moreover, for \(p\notin2\mathbb N\), we obtain a complete description of the linear isometric embeddings of the non-centered space \(R_\omega^p\), and, for \(p\neq2\), we determine its group of surjective linear isometries.
	\end{abstract}

\maketitle

\section{Introduction}

Let \(1<p<\infty\). In their seminal work \cite{BRS81}, Bourgain, Rosenthal, and
Schechtman introduced a transfinite family $(R_\alpha^p)_{\alpha<\omega_1}$ of complemented subspaces of \(L^p\). The construction is inherently distributional: successor steps are obtained by means of disjoint sums,
whereas limit steps are obtained by means of independent sums. The
first limit space, \(R_\omega^p\), is already closely connected with
the classical Rosenthal space \(X_p\). More precisely, $R_\omega^p\simeq X_p$,
see, for instance, \cite{A99,KM24}.

Recent work of Konstantos and Motakis \cite{KM24,KM25} has emphasized the operator-theoretic relevance of the distributional structure of the Bourgain--Rosenthal--Schechtman spaces (BRS spaces for short). They developed a framework of distributional embeddings for the centered spaces
\[
R_\alpha^{p,0}
=
\left\{
f\in R_\alpha^p:
\mathbb E[f]=0
\right\},
\]
constructed explicit finite-dimensional decompositions, and used distributional embeddings to reduce bounded operators on limit BRS spaces to scalar block-diagonal operators up to an arbitrarily small error. They also observed that their distributional embedding schemes do not exhaust all possible distributional embeddings and raised the problem of obtaining a combinatorial description of all such embeddings between BRS spaces, see \cite[Remark~5.7]{KM25}.

The purpose of the present note is to provide a complete answer to the corresponding self-embedding problem for the first limit centered space \(R_\omega^{p,0}\). We classify all its distributional self-embeddings and prove that it admits no proper non-zero internal compressions. As an application, for exponents \(p\notin2\mathbb N\), we also obtain a complete description of the linear isometric embeddings of the non-centered space \(R_\omega^p\). Moreover, for the larger range \(p\neq2\), we determine the group of surjective linear isometries of \(R_\omega^p\).

We work with the following concrete independent-sum realization. For
every \(n\in\mathbb N\), let $\Omega_n=\{0,1\}^n$
be equipped with the uniform probability measure, and set $H_n=L_p^0(\Omega_n)$.
On the product probability space $\Omega=\prod_{n=1}^{\infty}\Omega_n$,
we denote by $J_n:H_n\to L^p(\Omega)$
the canonical lifting to the \(n\)-th exterior coordinate. Then
\[
X_\omega:=
\overline{
	\operatorname{span}
	\{J_n(H_n):n\in\mathbb N\}
}^{\,L^p(\Omega)}
\]
is an isometric distributional realization of \(R_\omega^{p,0}\).

The first ingredient is an elementary rigidity principle for
finite-valued elements of \(X_\omega\). Every \(x\in X_\omega\) admits
a canonical exterior decomposition
\[
x=\sum_{n=1}^{\infty}J_nx_n,
\qquad
x_n\in H_n.
\]
We prove that, if \(x\) has at most \(q\) essential values, then
\[
\left|\left\{n\in\mathbb N:x_n\neq 0\right\}\right|\leq q-1.
\]
The proof relies on the independence of the exterior blocks, the
elementary inclusion $\supp(\mu)+\supp(\nu)\subseteq \supp(\mu*\nu)$ for Borel probability measures on \(\mathbb R\), and a cardinality
estimate for finite sumsets in the real line.

Passing to the unitalization $\widehat X_\omega=\mathbb R\mathbf 1\oplus X_\omega$,
we deduce a Boolean rigidity theorem: every non-trivial indicator
belonging to \(\widehat X_\omega\) depends on a unique exterior
coordinate. Equivalently, if $\mathbf 1_A\in\widehat X_\omega$,
then either \(A\) is trivial modulo null sets or there exist a unique
index \(n\in\mathbb N\) and a unique non-trivial subset $B\subseteq\Omega_n$
such that $\mathbf 1_A=J_n\mathbf 1_B$ almost everywhere.

This rigidity principle allows us to classify all distributional self-embeddings of \(X_\omega\). Recall that a linear operator $T:X_\omega \to X_\omega$ is called a distributional embedding if
\[
\dist(Tx)=
\dist(x)
\quad
\text{for every }x\in X_\omega.
\]
We prove that every such operator is determined by finite packing
data. More precisely, there exist a map $\kappa:\mathbb N\to\mathbb N$
and, for every \(k\in\mathbb N\), a measure-preserving factor
\[
\Phi_k:\Omega_k\longrightarrow\prod_{n\in I_k}\Omega_n,
\qquad
I_k=\kappa^{-1}(\{k\}),
\]
such that $\sum_{n\in I_k}n\leq k$ for every $k\in\mathbb N$. The corresponding embedding is given on finite exterior sums by
\[
T\left(\sum_{n=1}^{N}J_nf_n\right)
=\sum_{k=1}^{\infty}
J_k\left(\sum_{n\in I_k\cap\{1,\dots,N\}}
f_n\circ\pr_n\circ\Phi_k\right).
\]
Conversely, every admissible family of finite packing data defines a
distributional self-embedding of \(X_\omega\).

Thus, distributional self-embeddings of \(R_\omega^{p,0}\) cannot mix
the exterior blocks arbitrarily. Each source block is sent into a
unique target block. Several source blocks may share the same target
block only if they can be realized as jointly independent factors
inside the corresponding finite Bernoulli cube. The exact capacity
restriction is
\[
\sum_{\kappa(n)=k}n\leq k.
\]

As an application of this classification, we also study linear
isometries of the non-centered first limit space \(R_\omega^p\). In our
concrete realization, this space is identified with the unitalization $\widehat X_\omega$.
Assuming that $p\notin2\mathbb N$, we prove that every linear isometric embedding $U:R_\omega^p\to R_\omega^p$ is obtained, up to a global change of sign, by adjoining the constant functions to a distributional embedding $T:R_\omega^{p,0}\to R_\omega^{p,0}.$ More precisely, there exists $\varepsilon\in\{-1,1\}$ such that
\[
U(c\mathbf 1+x)=\varepsilon c\mathbf 1+\varepsilon Tx
\]
for every $c\in\mathbb R$ and $x\in R_\omega^{p,0}$. Thus, up to a global change of sign, the preceding finite packing
description also classifies all linear isometric embeddings of \(R_\omega^p\). For the larger range \(p\neq2\), we
determine the group of surjective linear isometries of \(R_\omega^p\). Namely, this group is isomorphic to
\[
\{-1,1\}\times\prod_{k=1}^{\infty}\mathfrak S_{2^k},
\]
where $\mathfrak S_{2^k}$ denotes the symmetric group on \(2^k\) elements.

Finally, following the terminology of Konstantos and Motakis \cite{KM25}, we study internal compressions. Let $0<\theta\leq 1$. A linear operator $S:X_\omega\to X_\omega$ is called a \(\theta\)-compression if
\[
\dist(Sx)=\theta\dist(x)+
(1-\theta)\delta_0
\quad
\text{for every }x\in X_\omega.
\]
We prove that \(X_\omega\) admits no proper non-zero internal compressions, that is, necessarily $\theta=1$.
The proof again relies on finite-valued rigidity. A hypothetical proper compression would force its entire range to be supported on a proper event depending on finitely many exterior coordinates. Such an event can support only a finite-dimensional subspace of \(X_\omega\), which contradicts injectivity.

The restriction to the first limit space is essential to the present argument. The space \(R_\omega^{p,0}\) admits a particularly transparent realization as an independent sum of finite Bernoulli blocks. For higher BRS spaces, the recursive interaction between disjoint and independent sums introduces an additional tree structure. A complete classification in that setting would require a corresponding tree-valued notion of packing.

The paper is organized as follows. In Section~2, we introduce the independent-sum realization of \(R_\omega^{p,0}\), establish its canonical exterior decomposition, prove the finite-valued and Boolean rigidity properties, and derive the finite-factor description of distributional embeddings. In Section~3, we prove the classification theorem and discuss some consequences. Section~4 is devoted to the classification of linear isometries of \(R_\omega^p\). Finally, Section~5 establishes the absence of proper internal compressions.

\section{The first limit BRS space as an independent sum}

Throughout the paper, we fix a real number $1<p<\infty$.
We use the notation $R_{\alpha}^{p,0}$ for the mean-zero Bourgain--Rosenthal--Schechtman spaces. We refer to \cite{BRS81} for the original construction, to \cite{A99} for the isomorphic theory of independent sums, and to \cite{KM25} for the distributional framework and the notation used below.

For every $n\in\mathbb N$, let $\Omega_n=\{0,1\}^n$
equipped with the uniform probability measure $\mu_n$, and define 
\[H_n=L_p^0(\Omega_n,\mu_n)=
\left\{ f\in L^p(\Omega_n,\mu_n):
	\int_{\Omega_n}f \; d\mu_n=0 \right\}.
\]
We consider the product probability space
\begin{equation}
	\label{Eq3}
(\Omega,\Sigma,\mu)=\prod_{n=1}^{\infty}
(\Omega_n,\mathcal P(\Omega_n),\mu_n)
\end{equation}
and denote by $\pi_n:\Omega\to\Omega_n$
the $n$-th coordinate projection. Naturally, the notation $\mc{P}(\O_n)$ in \eqref{Eq3} stands for the power set of $\O_n$. For every \(n\in\mathbb N\), let $J_n:L^p(\Omega_n,\mu_n)\to L^p(\Omega,\mu)$
be the coordinate lifting given by $J_nf=f\circ\pi_n$.
We use the same notation for its restriction to \(H_n\).
Note that the subspaces $J_n(H_n)$, \(n\in\mathbb N\), are independent and consist of mean-zero random variables.
We shall work with the following concrete realization of the first limit BRS space:
\[
X_\omega:=\overline{\operatorname{span}
	\left\{ J_n(H_n):n\in\mathbb N\right\}}^{L^p(\Omega)}.
\]
This space is a distributional realization of \(R_\omega^{p,0}\), so, in what follows, we identify $X_\omega=R_\omega^{p,0}$.
For each $n\in\mathbb N$, let $\Sigma_n=\sigma(\pi_n)$
be the finite sub-$\sigma$-algebra generated by the $n$-th exterior coordinate. 

\subsection{The canonical exterior decomposition}

For every \(N \in \mathbb{N}\), let $\mathcal{G}_N := \sigma(\pi_1,\dots,\pi_N)$ be the finite sub-\(\sigma\)-algebra generated by the first \(N\) exterior
coordinates, and denote by
\[
\mc{P}_N:=
\mathbb{E}[\;\cdot \mid \mathcal{G}_N]
\colon
L^p(\Omega)\longrightarrow L^p(\Omega,\mathcal{G}_N,\mu)
\]
the corresponding conditional expectation. We shall also use the coordinate conditional expectations $\mc{E}_n
:=\mathbb{E}[\; \cdot \mid \Sigma_n]$, $\Sigma_n := \sigma(\pi_n)$. The next proposition provides the canonical exterior decomposition of the elements of \(X_\omega\).

\begin{proposition}
	For every \(x \in X_\omega\), there exists a unique sequence $(x_n)_{n=1}^{\infty}$, $x_n \in H_n$, such that
	\begin{equation}
		\label{Eq2}
	x = \sum_{n=1}^{\infty} J_n x_n
\end{equation}
	with convergence in \(L^p(\Omega)\).
\end{proposition}

\begin{proof}
	We first consider an element of the algebraic span, $x = \sum_{m=1}^{M} J_m y_m$, $y_m \in H_m$.
	Fix \(n\in\mathbb N\). If \(1\leq n\leq M\), then \(J_ny_n\) is \(\Sigma_n\)-measurable and $\mc E_n(J_ny_n)=J_ny_n$. If \(m\neq n\), then \(J_my_m\) is independent of \(\Sigma_n\), and
	hence $\mc E_n(J_my_m)=0$. Therefore,
	\begin{equation}
		\label{Eq4}
	\mc{E}_n x=
	\begin{cases}
		J_n y_n, & 1 \leq n \leq M, \\[1ex]
		0, & n > M.
	\end{cases}
	\end{equation}
	Similarly, since \(\mathcal{G}_N\) is generated by the first \(N\)
	exterior coordinates,
	\[
	\mc{P}_N(J_m y_m)=
	\begin{cases}
		J_m y_m, & m \leq N, \\[1ex]
		0, & m > N.
	\end{cases}
	\]
	Therefore,
	\[
	\mc{P}_N x=\sum_{m=1}^{\min\{M,N\}} J_m y_m.
	\]
	We now consider an arbitrary element \(x \in X_\omega\). By the	definition of \(X_\omega\), there exists a sequence
	\((x^{(j)})_{j=1}^{\infty}\) in the algebraic span of the spaces \(J_n(H_n)\) such that
	\[
	\lim_{j \to \infty}
	\lVert x^{(j)} - x \rVert_{L^p(\Omega)}=0.
	\]
	For every \(n \in \mathbb{N}\), the conditional expectation \(\mc{E}_n\) is contractive. Hence, $\mc{E}_n x^{(j)}
	\to\mc{E}_n x$ in \(L^p(\Omega)\). For every \(j\), equation \eqref{Eq4} implies that
	\(\mc{E}_n x^{(j)} \in J_n(H_n)\). Since \(H_n\) is finite-dimensional, the subspace \(J_n(H_n)\) is closed in \(L^p(\Omega)\) and, therefore, $\mc{E}_n x \in J_n(H_n)$.
	Consequently, there exists a unique vector \(x_n \in H_n\) such that $J_n x_n = \mc{E}_n x$.
	For every \(N \in \mathbb{N}\), the identity
	\[
	\mc{P}_N x^{(j)}=\sum_{n=1}^{N} \mc{E}_n x^{(j)}
	\]
	holds on the algebraic span. Passing to the limit as \(j \to \infty\),
	we obtain
	\[
	\mc{P}_N x=\sum_{n=1}^{N} \mc{E}_n x=\sum_{n=1}^{N} J_n x_n.
	\]
	
	It remains to prove that \(\mc{P}_N x \to x\) in \(L^p(\Omega)\). Let \(\varepsilon > 0\) and choose an element \(y\) in the algebraic span of the spaces \(J_n(H_n)\) such that $\lVert x-y \rVert_{L^p(\Omega)} < \varepsilon$.
	There exists \(N_0 \in \mathbb{N}\) such that \(\mc{P}_N y = y\) for every \(N \geq N_0\). Since \(\mc{P}_N\) is contractive, for every \(N \geq N_0\) we have
	\[
	\begin{aligned}
		\lVert \mc{P}_N x - x \rVert_{L^p(\Omega)}
		&\leq
		\lVert \mc{P}_N(x-y) \rVert_{L^p(\Omega)}+
		\lVert \mc{P}_N y-y \rVert_{L^p(\Omega)}+
		\lVert y-x \rVert_{L^p(\Omega)} \\
		&\leq 2 \lVert x-y \rVert_{L^p(\Omega)} < 2\varepsilon.
	\end{aligned}
	\]
	Therefore, \(\mc{P}_N x \to x\) in \(L^p(\Omega)\), and hence \eqref{Eq2} holds. Finally, suppose that
	\[
	x=\sum_{n=1}^{\infty} J_n x_n=\sum_{n=1}^{\infty} J_n y_n
	\]
	are two convergent exterior decompositions with \(x_n,y_n \in H_n\). Applying \(\mc{E}_m\) to both identities yields $J_m x_m = J_m y_m$. Since \(J_m\) is injective, \(x_m = y_m\). Thus, the decomposition is
	unique.
\end{proof}

The preceding proposition allows us to introduce the following notation.

\begin{definition}
	Let $x = \sum_{n=1}^{\infty} J_n x_n \in X_\omega$
	be the canonical exterior decomposition of \(x\). The exterior support
	of \(x\) is the set
	\[
	\supp_{\mathrm{ext}}(x):=
	\{n \in \mathbb{N} : x_n \neq 0\}.
	\]
\end{definition}

\subsection{Finite-valued elements have finite exterior support}

For a real-valued random variable \(f\) on a probability space $(\O,\mu)$, we denote by
$\operatorname{dist}(f):=f_{\#}\mu$ its distribution, that is, the Borel probability measure on
\(\mathbb R\) defined by
\[
\operatorname{dist}(f)(B)=\mu\bigl(f^{-1}(B)\bigr)
\]
for every Borel set \(B\subseteq\mathbb R\). We denote by
\(\essran(f)\) its essential range, namely the set of all \(a\in\mathbb R\) such that $\mu\bigl(\{|f-a|<\varepsilon\}\bigr) > 0$ for every $\varepsilon>0$. Equivalently, 
\[ 
\essran(f) = \supp\bigl(\dist(f)\bigr),
\]
where \(\supp(\nu)\) denotes the topological support of a Borel probability measure \(\nu\) on \(\mathbb R\), that is, the set of all \(a\in\mathbb R\) such that $\nu\bigl((a-\varepsilon,a+\varepsilon)\bigr) > 0$ for every $\varepsilon>0$. Indeed, 
\[ 
\dist(f) \bigl( (a-\varepsilon,a+\varepsilon) \bigr) = \mu\bigl(\{|f-a|<\varepsilon\}\bigr). 
\]

We first recall two elementary facts and include their proofs for completeness.

\begin{lemma}\label{l2.1}
	Let \(\mu\) and \(\nu\) be Borel probability measures on \(\mathbb{R}\).
	Then
	\[
	\supp(\mu)+\supp(\nu)
	\subseteq \supp(\mu * \nu).
	\]
\end{lemma}

\begin{proof}
	Let \(a \in \supp(\mu)\) and
	\(b \in \supp(\nu)\). We prove that
	\(a+b \in \supp(\mu * \nu)\). Let \(U\) be an open neighborhood of \(a+b\) and choose open
	neighborhoods \(V\) of \(a\) and \(W\) of \(b\) such that
	\(V+W \subseteq U\). Since \(a \in \supp(\mu)\) and
	\(b \in \supp(\nu)\), we have \(\mu(V)>0\) and
	\(\nu(W)>0\). Therefore,
	\[
	\begin{aligned}
		(\mu * \nu)(U)
		=
		\int_{\mathbb{R}} \mu(U-y) \, d\nu(y)
		\geq
		\int_W \mu(U-y) \, d\nu(y)
		\geq
		\mu(V)\nu(W)
		>
		0.
	\end{aligned}
	\]
	Hence \(a+b \in \supp(\mu * \nu)\).
\end{proof}

\begin{lemma}\label{l2.2}
	Let \(A_1,\dots,A_m\) be non-empty finite subsets of \(\mathbb{R}\).
	Then
	\[
	\lvert A_1+\cdots+A_m \rvert
	\geq 1+\sum_{j=1}^m \bigl(\lvert A_j\rvert-1\bigr).
	\]
	In particular, if \(\lvert A_j\rvert \geq 2\) for every
	\(j \in \{1,\dots,m\}\), then
	\[
	\lvert A_1+\cdots+A_m \rvert \geq m+1.
	\]
\end{lemma}

\begin{proof}
	It is enough to prove that \(\lvert A+B\rvert \geq \lvert A\rvert+\lvert B\rvert-1\) for every pair of non-empty finite subsets \(A,B \subseteq \mathbb{R}\). Write
	\[
	A=\{a_1<\cdots<a_r\}, \qquad B=\{b_1<\cdots<b_s\}.
	\]
	The following \(r+s-1\) elements of \(A+B\) are strictly increasing:
	\[
	a_1+b_1<a_2+b_1
	<\cdots<
	a_r+b_1<a_r+b_2<\cdots<
	a_r+b_s.
	\]
	Thus, \(\lvert A+B\rvert \geq r+s-1=\lvert A\rvert+\lvert B\rvert-1\).
	The general statement follows by induction on \(m\).
\end{proof}

We can now prove the finite-support property for the canonical exterior
decomposition.

\begin{proposition}\label{p2.2}
	Let $x=\sum_{n=1}^{\infty} J_n x_n \in X_\omega$ be the canonical exterior decomposition of \(x\). Assume that
	\(\essran(x)\) is finite and satisfies \(\lvert \essran(x)\rvert \leq q\). Then $\supp_{\mathrm{ext}}(x)$ is finite and
	\[
	\bigl\lvert \supp_{\mathrm{ext}}(x) \bigr\rvert \leq q-1.
	\]
\end{proposition}

\begin{proof}
	If  $\supp_{\mathrm{ext}}(x) = \varnothing$, there is nothing to prove. Otherwise, let \(F \subseteq \supp_{\mathrm{ext}}(x)\) be a non-empty finite subset, and write \(F=\{n_1,\dots,n_m\}\). Define
	\[
	S_F := \sum_{j=1}^m J_{n_j}x_{n_j}, \qquad R_F := x-S_F.
	\]
	Note that the random variable \(S_F\) depends only on the exterior coordinates
	indexed by \(F\). On the other hand, by the canonical exterior decomposition,
	\[
	R_F=\sum_{n\notin F} J_n x_n
	\]
	with convergence in \(L^p(\Omega)\). Each partial sum of $R_F$ is measurable with respect to the $\sigma$-algebra $\mc{A}_{F}:=\sigma(\pi_n:n\notin F)$. Since $L^p\bigl(\Omega,\mc{A}_{F},\mu\bigr)$ is a closed subspace of \(L^p(\Omega)\), \(R_F\) is measurable with respect to \(\mc{A}_{F}\). Moreover, as the corresponding
	families of exterior coordinates are independent, the random variables
	\(S_F\) and \(R_F\) are independent. For every \(j \in \{1,\dots,m\}\), let
	\[
	A_j := \supp\bigl(\dist(x_{n_j})\bigr).
	\]
	Since \(x_{n_j}\) is defined on the finite probability space \(\Omega_{n_j}\), the set \(A_j\) is finite. Moreover, as \(x_{n_j} \neq 0\) and \(x_{n_j} \in H_{n_j}\), the random
	variable \(x_{n_j}\) is not constant, indeed, the only constant element
	of \(H_{n_j}\) is the zero function. Therefore, \(\lvert A_j\rvert \geq 2\). 
	The random variables $J_{n_1}x_{n_1}, \dots, J_{n_m}x_{n_m}$ are independent, since they depend on distinct exterior coordinates. Moreover, $\dist(J_{n_j}x_{n_j}) = \dist(x_{n_j})$ for every \(j\). So, by independence,
	\[
	\dist(S_F) = \dist(x_{n_1}) *\cdots *\dist(x_{n_m}).
	\]
	Applying Lemma~\ref{l2.1} repeatedly, we obtain
	\[
	A_1+\cdots+A_m\subseteq\supp\bigl(\dist(S_F)\bigr).
	\]
	Choose $b \in \supp\bigl(\dist(R_F)\bigr)$. Such a point exists because \(\dist(R_F)\) is a probability measure. Since \(S_F\) and \(R_F\) are independent,
	\[
	\dist(x)=\dist(S_F)*\dist(R_F).
	\]
	Another application of Lemma~\ref{l2.1} yields
	\[
	A_1+\cdots+A_m+b\subseteq\supp\bigl(\dist(x)\bigr)=\essran(x).
	\]
	We may therefore apply Lemma~\ref{l2.2}. Since
	\(\lvert A_j\rvert \geq 2\) for every \(j\), we obtain
	\[
	q \geq	\lvert \essran(x)\rvert\geq\lvert A_1+\cdots+A_m+b\rvert
	=\lvert A_1+\cdots+A_m\rvert\geq m+1.
	\]
	Thus, \(m\leq q-1\). Since \(F\) was an arbitrary finite subset of
	\(\supp_{\mathrm{ext}}(x)\), we conclude that $\bigl\lvert \supp_{\mathrm{ext}}(x) \bigr\rvert\leq q-1$.
\end{proof}

\subsection{Boolean rigidity in the unitalization}

We now enlarge \(X_\omega\) by adjoining the constant functions.

\begin{definition}
	The unitalization of \(X_\omega\) is the space
	\[
	\widehat{X}_\omega:=\mathbb{R}\mathbf{1}
	\oplus X_\omega \subseteq L^p(\Omega).
	\]
	Thus, every \(g \in \widehat{X}_\omega\) admits a unique representation
	\[
	g=	c\mathbf{1} +\sum_{n=1}^{\infty} J_n x_n,
	\]
	where \(c \in \mathbb{R}\), \(x_n \in H_n\), and the series converges
	in \(L^p(\Omega)\).
\end{definition}

The next result is the Boolean rigidity property that will drive the
classification of distributional embeddings.

\begin{theorem}[Boolean rigidity]\label{t2.1}
	Let \(A \in \Sigma\) be a measurable set and assume that
	\(\mathbf{1}_A \in \widehat{X}_\omega\). Then exactly one of the
	following alternatives holds:
	\begin{enumerate}
		\item[{\rm (i)}] \(A=\varnothing\) modulo null sets;
		
		\item[{\rm (ii)}] \(A=\Omega\) modulo null sets;
		
		\item[{\rm (iii)}] there exist a unique index \(n_0 \in \mathbb{N}\) and a
		unique non-trivial subset \(B \subseteq \Omega_{n_0}\) such that
		\[
		\mathbf{1}_A=J_{n_0}\mathbf{1}_B.
		\]
	\end{enumerate}
	Equivalently, every non-trivial indicator belonging to
	\(\widehat{X}_\omega\) depends on a unique exterior coordinate.
\end{theorem}

\begin{proof}
	Write
	\[
	\mathbf{1}_A=c\mathbf{1}+\sum_{n=1}^{\infty} J_n x_n,
	\]
	where \(c \in \mathbb{R}\) and \(x_n \in H_n\). Since \(\mathbf{1}_A\) takes only the values \(0\) and \(1\), the
	function $\mathbf{1}_A-c\mathbf{1}$
	has finite essential range. Indeed,
	\[
	\essran\bigl(\mathbf{1}_A-c\mathbf{1}\bigr)
	\subseteq \{-c,1-c\}.
	\]
	By Proposition~\ref{p2.2}, at
	most one exterior component \(x_n\) can be non-zero. If \(x_n=0\) for every \(n\in\N\), then \(\mathbf{1}_A=c\mathbf{1}\). Since
	\(\mathbf{1}_A\) is an indicator, \(c \in \{0,1\}\). Hence
	\(A=\varnothing\) or \(A=\Omega\) modulo null sets. Assume now that there exists an index \(n_0 \in \mathbb{N}\) such that
	\(x_{n_0}\neq 0\). Then
	\[
	\mathbf{1}_A =c\mathbf{1}+J_{n_0}x_{n_0}.
	\]
	Define $h := c\mathbf 1+x_{n_0}$ as a function on \(\Omega_{n_0}\). Then $\mathbf 1_A = J_{n_0}h=h\circ \pi_{n_0}$ almost everywhere. Since every atom of \(\Omega_{n_0}\) has positive measure and \(J_{n_0}h\) is almost everywhere an indicator, we have $h(a)\in\{0,1\}$ for every $a\in\Omega_{n_0}$. Set 
	\[ 
	B := \{a\in\Omega_{n_0}:h(a)=1\}. 
	\] 
	Then $h = \mathbf 1_B$, and therefore $\mathbf 1_A = J_{n_0}\mathbf 1_B$ almost everywhere. The set \(B\) is non-trivial. Indeed, if \(B=\varnothing\) or \(B=\Omega_{n_0}\), then \(\mathbf 1_A\) would be constant, contradicting \(x_{n_0}\neq 0\).
	
	It remains to prove the uniqueness of \(n_0\). Suppose that there exist
	indices \(m\neq n\) and subsets \(B_m \subseteq \Omega_m\) and
	\(B_n \subseteq \Omega_n\) such that
	\[
	J_m\mathbf{1}_{B_m}=J_n\mathbf{1}_{B_n}
	\]
	almost everywhere, and assume that this common indicator is
	non-trivial. The random variables \(J_m\mathbf{1}_{B_m}\) and
	\(J_n\mathbf{1}_{B_n}\) are independent, because they depend on
	different exterior coordinates. Since they are equal almost
	everywhere, the common random variable is independent of itself. Let $a:=\mu_m(B_m)=\mu_n(B_n)$.
	Then
	\[
	a=	\mu\bigl(J_m\mathbf{1}_{B_m}=1\bigr).
	\]
	By independence and almost-everywhere equality,
	\[
	\begin{aligned}
		a=\mu\bigl(J_m\mathbf{1}_{B_m}=1,\, J_n\mathbf{1}_{B_n}=1\bigr)=a^2.
	\end{aligned}
	\]
	Hence \(a\in\{0,1\}\), which contradicts the non-triviality of the
	common indicator. Finally, once \(n_0\) is fixed, the subset \(B\) is unique because
	\(J_{n_0}\) is injective.
\end{proof}

\subsection{Distributional embeddings and finite probability factors}

We now introduce the class of operators that will be studied throughout
the paper.

\begin{definition}
	A linear operator \(T \colon X_\omega \to X_\omega\) is
	called a \emph{distributional embedding} if
	\[
	\dist(Tx)=\dist(x)\quad\text{for every } x \in X_\omega.
	\]
\end{definition}
Every distributional embedding is an isometry. Indeed, $\lVert Tx \rVert_{L^p(\Omega)}=\lVert x \rVert_{L^p(\Omega)}$ for every $x \in X_\omega$.
In particular, every distributional embedding is injective. The next observation allows us to work with indicators in the
unitalization.

\begin{lemma}\label{l2.3}
	Let \(T \colon X_\omega \to X_\omega\) be a distributional embedding. Define $\widehat{T}\colon\widehat{X}_\omega\to\widehat{X}_\omega$ by
	\[
	\widehat{T}(c\mathbf{1}+x):=c\mathbf{1}+Tx,
	\qquad	c \in \mathbb{R}, \quad x\in X_\omega.
	\]
	Then \(\widehat{T}\) is a well-defined linear operator and $\dist(\widehat{T}y)=\dist(y)$ for every $y \in \widehat{X}_\omega$.
\end{lemma}

\begin{proof}
	The decomposition $\widehat{X}_\omega=\mathbb{R}\mathbf{1}\oplus X_\omega$
	is direct, because every element of \(X_\omega\) has mean zero.
	Therefore, \(\widehat{T}\) is well defined and linear.
	
	Let \(y=c\mathbf{1}+x\), where \(c \in \mathbb{R}\) and
	\(x \in X_\omega\). Since \(T\) preserves distributions, we have
	\(\dist(Tx)=\dist(x)\). Translation by the
	constant \(c\) yields
	\[
	\dist(c\mathbf{1}+Tx)=\dist(c\mathbf{1}+x).
	\]
	Thus, \(\dist(\widehat{T}y)=\dist(y)\).
\end{proof}

We next study the restriction of a distributional embedding to a fixed
finite exterior block. Fix \(n \in \mathbb{N}\). For every atom \(a \in \Omega_n\), define
\[
u_a:= \mathbf{1}_{\{a\}}-2^{-n}\mathbf{1}\in H_n.
\]
We identify \(u_a\) with its canonical copy
\(J_nu_a \in J_n(H_n) \subseteq X_\omega\).

\begin{lemma}\label{l2.4}
	Let \(T \colon X_\omega \to X_\omega\) be a distributional
	embedding. For every \(n \in \mathbb{N}\) and every \(a \in \Omega_n\),
	the function
	\[
	b_a:=\widehat{T}\bigl(J_n\mathbf{1}_{\{a\}}\bigr)=
	T(J_nu_a)+2^{-n}\mathbf{1}
	\]
	is an indicator of measure \(2^{-n}\). Moreover,
	\[
	\sum_{a \in \Omega_n} b_a= \mathbf{1}.
	\]
	Consequently, there exist measurable sets \(B_a \subseteq \Omega\),
	indexed by \(a \in \Omega_n\), which form a measurable partition of
	\(\Omega\) modulo null sets and satisfy $b_a=\mathbf{1}_{B_a}$ and $\mu(B_a)=2^{-n}$.
\end{lemma}

\begin{proof}
	Since $J_n\mathbf{1}_{\{a\}}=J_nu_a+2^{-n}\mathbf{1}$,
	Lemma~\ref{l2.3} gives
	\[
	\dist(b_a)=\dist\bigl(J_n\mathbf{1}_{\{a\}}\bigr).
	\]
	The function \(J_n\mathbf{1}_{\{a\}}=\mathbf{1}_{\pi_{n}^{-1}(\{a\})}\) is an indicator of measure
	\(2^{-n}\). Therefore,
	\[
	\dist(b_a)=2^{-n}\delta_1+(1-2^{-n})\delta_0.
	\]
	Hence \(b_a\) is the indicator of a measurable set \(B_a \subseteq \Omega\) satisfying \(\mu(B_a)=2^{-n}\). 
	Since the singletons \(\{a\}\), \(a\in\Omega_n\), form a partition of \(\Omega_n\), we have
	\[
	\sum_{a\in\Omega_n}\mathbf 1_{\{a\}}=\mathbf 1
	\]
	as functions on \(\Omega_n\). By linearity of the coordinate embedding \(J_n\), and since \(J_n\mathbf 1=\mathbf 1\), it follows that
	\[
	\sum_{a\in\Omega_n}J_n\mathbf 1_{\{a\}}=
	J_n\left( \sum_{a\in\Omega_n}\mathbf 1_{\{a\}} \right)
	=J_n\mathbf 1=\mathbf 1.
	\]
	Since \(\widehat{T}\mathbf{1}=\mathbf{1}\), linearity again yields
	\[
	\sum_{a \in \Omega_n} b_a
	=
	\mathbf{1}.
	\]
	As each \(b_a\) is an indicator, the sets \(B_a\), indexed by \(a \in \Omega_n\), form a measurable partition of \(\Omega\) modulo null sets.
\end{proof}

The Boolean rigidity theorem forces the entire finite block \(J_n(H_n)\) to land inside a single exterior block.

\begin{proposition}\label{p2.3}
	Let \(T \colon X_\omega \to X_\omega\) be a distributional
	embedding. Then, for every \(n \in \mathbb{N}\), there exists a unique
	index \(\kappa(n) \in \mathbb{N}\) such that
	\begin{equation}
		\label{Eq1}
	T\bigl(J_n(H_n)\bigr) \subseteq J_{\kappa(n)}(H_{\kappa(n)}).
	\end{equation}
	Moreover, \(\kappa(n)\geq n\).
\end{proposition}

\begin{proof}
	Fix \(n \in \mathbb{N}\). By Lemma~\ref{l2.4}, for every \(a \in \Omega_n\)
	there exists an indicator \(b_a=\mathbf{1}_{B_a} \in \widehat{X}_\omega\)
	of measure \(2^{-n}\), and the sets \(B_a\), indexed by
	\(a \in \Omega_n\), form a measurable partition of \(\Omega\).
	
	Since \(0<\mu(B_a)=2^{-n}<1\),
	Theorem~\ref{t2.1} implies that, for every
	\(a \in \Omega_n\), there exist an index \(k(a) \in \mathbb{N}\) and a
	non-trivial subset \(C_a \subseteq \Omega_{k(a)}\) such that
	\[
	\mathbf{1}_{B_a}= J_{k(a)}\mathbf{1}_{C_a}.
	\]
	
	We claim that \(k(a)\) does not depend on \(a\). Assume, to the
	contrary, that there exist distinct atoms \(a,a' \in \Omega_n\) such
	that \(k(a)\neq k(a')\). Then \(B_a\) and \(B_{a'}\) depend on different
	exterior coordinates and so, they are independent events. Since both
	have positive measure,
	\[
	\mu(B_a \cap B_{a'})=\mu(B_a)\mu(B_{a'})=2^{-2n}>0.
	\]
	This contradicts the fact that the sets \(B_a\), indexed by
	\(a \in \Omega_n\), form a partition.
	
	Therefore, there exists an index \(\kappa(n) \in \mathbb{N}\) such that $\mathbf{1}_{B_a}=J_{\kappa(n)}\mathbf{1}_{C_a}$ for every $a \in \Omega_n$,
	for suitable subsets \(C_a \subseteq \Omega_{\kappa(n)}\). Since
	\(\sum_{a \in \Omega_n}\mathbf{1}_{B_a}=\mathbf{1}\), the sets
	\(C_a\), indexed by \(a \in \Omega_n\), form a partition of
	\(\Omega_{\kappa(n)}\) modulo null sets. Since \(\Omega_{\kappa(n)}\) is finite and every atom has positive
	measure, this is an actual partition. Moreover, $\mu_{\kappa(n)}(C_a)
	=2^{-n}$ for every $a \in \Omega_n$.
	Every atom of \(\Omega_{\kappa(n)}\) has measure
	\(2^{-\kappa(n)}\). Therefore, \(2^{-n}\) must be an integer multiple of \(2^{-\kappa(n)}\). Equivalently,
	\(2^{\kappa(n)-n} \in \mathbb{N}\), and hence \(\kappa(n)\geq n\).
	
	Finally, the centered indicators
	\[
	u_a=\mathbf{1}_{\{a\}}-2^{-n}\mathbf{1},
	\qquad a \in \Omega_n,
	\]
	span \(H_n\). For every \(a \in \Omega_n\),
	\[
	\begin{aligned}
		T(J_nu_a)=
		\mathbf{1}_{B_a}-2^{-n}\mathbf{1}=
		J_{\kappa(n)}\bigl(\mathbf{1}_{C_a}-2^{-n}\mathbf{1}\bigr).
	\end{aligned}
	\]
	Thus, \eqref{Eq1} holds.
	
	It remains to prove the uniqueness of \(\kappa(n)\). Suppose that
	there exists an index \(\ell\in\mathbb N\) such that $T\bigl(J_n(H_n)\bigr) \subseteq J_\ell(H_\ell)$. Fix \(a\in\Omega_n\). Then $T(J_nu_a) \in J_\ell(H_\ell)$,
	and hence
	\[
	b_a=T(J_nu_a)+2^{-n}\mathbf 1
	\]
	is a non-trivial indicator depending only on the exterior coordinate
	\(\ell\). On the other hand, $b_a=J_{\kappa(n)}\mathbf 1_{C_a}$.
	The uniqueness statement in
	Theorem~\ref{t2.1} therefore yields $\ell=\kappa(n)$.
	This proves the uniqueness of \(\kappa(n)\).
\end{proof}

The preceding proof yields a concrete description of the restriction
of \(T\) to each finite exterior block.

\begin{corollary}\label{c2.1}
	Let \(T \colon X_\omega \to X_\omega\) be a distributional
	embedding. For every \(n \in \mathbb{N}\), there exists a
	measure-preserving map $\varphi_n\colon\Omega_{\kappa(n)} \to\Omega_n$
	such that
	\[
	T(J_nf)=J_{\kappa(n)}(f \circ \varphi_n) \qquad \text{for every } f \in H_n.
	\]
	Equivalently, the restriction of \(T\) to the block \(J_n(H_n)\) is
	the pullback induced by a finite probability factor.
\end{corollary}

\begin{proof}
	For each \(a \in \Omega_n\), let
	\(C_a \subseteq \Omega_{\kappa(n)}\) be the set obtained in the proof
	of Proposition~\ref{p2.3}. The sets \(C_a\), indexed by \(a \in \Omega_n\), form a partition of \(\Omega_{\kappa(n)}\), and $\mu_{\kappa(n)}(C_a)=2^{-n}$.
	Define $\varphi_n\colon\Omega_{\kappa(n)}\to\Omega_n$ by setting \(\varphi_n(\omega)=a\) whenever \(\omega \in C_a\). Then \(\varphi_n^{-1}(\{a\})=C_a\), and hence
	\[
	\mu_{\kappa(n)}\bigl(\varphi_n^{-1}(\{a\})\bigr)=2^{-n}=\mu_n(\{a\}).
	\]
	
	Since the singletons generate the \(\sigma\)-algebra of \(\Omega_n\),
	the map \(\varphi_n\) preserves the uniform probability measure. For every \(a \in \Omega_n\), we have
	\[
	\begin{aligned}
		u_a\circ\varphi_n=\mathbf 1_{\varphi_n^{-1}(\{a\})}
		-2^{-n}\mathbf 1=\mathbf 1_{C_a}-2^{-n}\mathbf 1.
	\end{aligned}
	\]
	Therefore, by Proposition~\ref{p2.3},
	\[
	T(J_nu_a)=J_{\kappa(n)}(u_a\circ\varphi_n).
	\]
	Since the vectors \(u_a\), indexed by \(a \in \Omega_n\), span \(H_n\),
	the identity extends by linearity to every \(f \in H_n\).
\end{proof}

\subsection{Compatible packings inside a target block}

Let \(T \colon X_\omega \to X_\omega\) be a distributional
embedding. Recall that by Proposition~\ref{p2.3},
there exists a map $\kappa \colon \mathbb{N} \to \mathbb{N}$
such that
\[
T\bigl(J_n(H_n)\bigr)\subseteq J_{\kappa(n)}(H_{\kappa(n)})
\]
for every \(n \in \mathbb{N}\). Moreover, by Corollary~\ref{c2.1}, for every
\(n \in \mathbb{N}\) there exists a measure-preserving map $\varphi_n\colon \Omega_{\kappa(n)}\to\Omega_n$
such that
\[
T(J_nf) =J_{\kappa(n)}(f \circ \varphi_n)
\quad \text{for every } f \in H_n.
\]
For every \(k \in \mathbb{N}\), define the fibre
\[
I_k:=\kappa^{-1}(\{k\})=\{n \in \mathbb{N} : \kappa(n)=k\}.
\]

Several source blocks may be mapped into the same target block. We now
prove that the corresponding finite probability factors must be
jointly independent. In the sequel, we use the convention that the
product indexed by the empty set is the one-point probability space.

\begin{lemma}\label{l2.5}
	Fix \(k \in \mathbb{N}\), and let \(F \subseteq I_k\) be finite. Then
	the family of maps \((\varphi_n)_{n \in F}\) is jointly independent. Equivalently, the map
	\[
	\Phi_{k,F}\colon\Omega_k\longrightarrow\prod_{n \in F} \Omega_n,\qquad\Phi_{k,F}(\omega):=\bigl(\varphi_n(\omega)\bigr)_{n \in F},
	\]
	preserves the product probability measure.
\end{lemma}

\begin{proof}
	If \(F=\varnothing\), the conclusion is immediate. Assume that $F=\{n_1,\dots,n_r\}$
	with \(r\in\N\). For every \(j \in \{1,\dots,r\}\), choose an arbitrary function \(f_j \in H_{n_j}\). The source random variables $J_{n_1}f_1,\dots,J_{n_r}f_r$ are independent, because they depend on distinct exterior coordinates.
	Since \(T\) is linear and preserves distributions, for every choice of
	scalars \(t_1,\dots,t_r \in \mathbb{R}\) we have
	\[
	\dist\left(\sum_{j=1}^r t_j T(J_{n_j}f_j)\right)
	=\dist\left(\sum_{j=1}^r t_j J_{n_j}f_j\right).
	\]
	Since \(\kappa(n_j)=k\) for every \(j\), Corollary~\ref{c2.1}
	gives $T(J_{n_j}f_j)=J_k(f_j \circ \varphi_{n_j})$.
	Therefore,
	\[
	\dist\left(\sum_{j=1}^r t_j f_j \circ \varphi_{n_j}\right)
	=\dist\left(\sum_{j=1}^r t_j f_j\right),
	\]
	where the sum on the right-hand side is interpreted on the product
	probability space $\prod_{j=1}^r\Omega_{n_j}$. By the Cramér--Wold theorem (see, for instance, \cite[Cr. 6.5]{K21}), the random vector $\bigl(f_1\circ\varphi_{n_1},\dots,f_r\circ\varphi_{n_r}\bigr)$
	on \(\Omega_k\) has the same distribution as the random vector $(f_1,\dots,f_r)$
	on the product probability space $\prod_{j=1}^r \Omega_{n_j}$. Now choose arbitrary subsets \(A_j \subseteq \Omega_{n_j}\), for \(j=1,\dots,r\), and set
	\[
	f_j:=\mathbf{1}_{A_j}-\mu_{n_j}(A_j)\mathbf{1}.
	\]
	Then \(f_j \in H_{n_j}\). The equality of joint distributions implies
	\[
	\mu_k\left(\bigcap_{j=1}^r\varphi_{n_j}^{-1}(A_j)\right)
	=\prod_{j=1}^r\mu_{n_j}(A_j).
	\]
	Indeed, the event $\left\{f_j \circ \varphi_{n_j}=1-\mu_{n_j}(A_j)\right\}$
	coincides with \(\varphi_{n_j}^{-1}(A_j)\), and similarly for the source variables. Thus, the sub-\(\sigma\)-algebras $\varphi_n^{-1}\bigl(\mathcal{P}(\Omega_n)\bigr)$, $n \in F$, are jointly independent. Equivalently, \(\Phi_{k,F}\) preserves the product probability measure.
\end{proof}

The finite-dimensional structure of the cubes now gives a sharp
packing restriction.

\begin{proposition}\label{p2.4}
	For every \(k \in \mathbb{N}\), the fibre \(I_k=\kappa^{-1}(\{k\})\)
	is finite and satisfies
	\[
	\sum_{n \in I_k} n\leq k.
	\]
\end{proposition}

\begin{proof}
	Let \(F \subseteq I_k\) be finite. By
	Lemma~\ref{l2.5}, the map
	\[
	\Phi_{k,F}\colon\Omega_k\longrightarrow\prod_{n \in F} \Omega_n
	\]
	preserves the product probability measure. The domain \(\Omega_k=\{0,1\}^k\) has cardinality \(2^k\), whereas the
	target space has cardinality
	\[
	\left\lvert\prod_{n \in F} \Omega_n
	\right\rvert=\prod_{n \in F} 2^n=2^{\sum_{n \in F} n}.
	\]
	Every atom of the target product space has measure
	\(2^{-\sum_{n \in F}n}\). Since \(\Phi_{k,F}\) preserves the product
	probability measure, the preimage of every target atom has positive
	measure and is therefore non-empty. Hence \(\Phi_{k,F}\) is
	surjective. Consequently,
	\[
	2^{\sum_{n \in F} n}\leq2^k,
	\]
	and therefore $\sum_{n \in F} n\leq k$.
	
	This inequality holds for every finite subset \(F \subseteq I_k\).
	Since every \(n \in I_k\) satisfies \(n \geq 1\), the fibre \(I_k\)
	cannot contain more than \(k\) elements. Hence \(I_k\) is finite.
	Taking \(F=I_k\), we obtain $\sum_{n \in I_k} n\leq k$.
	The proof is complete.
\end{proof}

For every \(k \in \mathbb{N}\), the preceding proposition allows us to
define the joint factor
\[
\Phi_k\colon\Omega_k
\longrightarrow\prod_{n \in I_k} \Omega_n, \qquad \Phi_k(\omega) :=\bigl(\varphi_n(\omega)\bigr)_{n \in I_k}.
\]

\begin{corollary}\label{c2.2}
	For every \(k \in \mathbb{N}\), the map $\Phi_k$ preserves the product probability measure. Moreover, $\sum_{n \in I_k} n\leq k$.
\end{corollary}

\begin{proof}
	By Proposition~\ref{p2.4}, the fibre \(I_k\) is finite and satisfies $\sum_{n\in I_k}n \leq k$.
	Applying Lemma~\ref{l2.5} with \(F=I_k\), we conclude that \(\Phi_k\) preserves the product probability measure.
\end{proof}

\section{Classification of distributional self-embeddings of \(X_\omega\)}

We now assemble the preceding finite-dimensional reductions into a
complete description of the distributional self-embeddings of
\(X_\omega=R_\omega^{p,0}\).

\begin{theorem}\label{t3.1}
	Let \(T \colon X_\omega \to X_\omega\) be a bounded linear
	operator. Then, \(T\) is a distributional embedding if and only if there exist
		\begin{enumerate}
			\item[{\rm (i)}] a map \(\kappa \colon \mathbb{N} \to \mathbb{N}\) such that, for every \(k \in \mathbb{N}\), the fibre $I_k:=\kappa^{-1}(\{k\})$ is finite,
			
			\item[{\rm (ii)}]  for every \(k \in \mathbb{N}\), a measure-preserving map
			\[
			\Phi_k\colon\Omega_k\longrightarrow\prod_{n \in I_k}\Omega_n,
			\]
		\end{enumerate}
		such that $\sum_{n \in I_k} n\leq k$ for every  $k \in \mathbb{N}$, and
		\begin{equation}
		\label{Eq9}
		T\left(\sum_{n=1}^{N} J_n f_n\right)=\sum_{k=1}^{\infty}
		J_k\left(\sum_{n \in I_k \cap \{1,\dots,N\}}f_n\circ\pr_n
		\circ\Phi_k\right)
		\end{equation}
		for every \(N \in \mathbb{N}\) and every finite family
		\(f_1,\dots,f_N\) with \(f_n \in H_n\), where 
		for each \(n\in I_k\), we denote by $\pr_n:
		\prod_{m\in I_k}\Omega_m\to\Omega_n$
		the canonical coordinate projection.
		In the formula above, only finitely many terms are non-zero.
\end{theorem}

\begin{proof}
	Assume first that \(T\) is a distributional embedding. This implication is a direct consequence of the preceding results. By Proposition~\ref{p2.3}, for every \(n \in \mathbb{N}\) there exists a unique index \(\kappa(n) \in \mathbb{N}\) such that
	\[
	T\bigl(J_n(H_n)\bigr)\subseteq J_{\kappa(n)}(H_{\kappa(n)}).
	\]
	By Corollary~\ref{c2.1}, for every \(n \in \mathbb{N}\)
	there exists a measure-preserving map $\varphi_n\colon\Omega_{\kappa(n)}\to\Omega_n$
	such that $T(J_nf)=J_{\kappa(n)}(f\circ\varphi_n)$ for every $f\in H_n$. For every \(k \in \mathbb{N}\), defining
	\(I_k=\kappa^{-1}(\{k\})\), by Proposition~\ref{p2.4}, the set \(I_k\) is finite
	and satisfies $\sum_{n \in I_k} n\leq k$. By Corollary~\ref{c2.2}, the map
	\[
	\Phi_k\colon\Omega_k \longrightarrow\prod_{n \in I_k}\Omega_n, \qquad \Phi_k:=(\varphi_n)_{n \in I_k},
	\]
	preserves the product probability measure. Note that, in particular, $\varphi_n =
	\pr_n\circ\Phi_k$ whenever $n\in I_k$.
	Now, let $x=\sum_{n=1}^{N} J_n f_n$, $f_n \in H_n$.
	Using linearity and the preceding description of the restrictions of
	\(T\), we obtain
	\[
	\begin{aligned}
		Tx&=\sum_{n=1}^{N} T(J_nf_n)
		=\sum_{n=1}^{N}J_{\kappa(n)}(f_n\circ\varphi_n)=
		\sum_{k=1}^{\infty}J_k\left(\sum_{n \in I_k \cap \{1,\dots,N\}}
		f_n\circ\pr_n
		\circ\Phi_k\right).
	\end{aligned}
	\]
	This proves the necessity of the stated representation.

	Conversely, assume that there exist data $\kappa$, $(I_k)_{k \in \mathbb{N}}$ and $(\Phi_k)_{k \in \mathbb{N}}$ satisfying the conditions in the statement and that \(T\) satisfies the stated formula \eqref{Eq9} on finite exterior sums. We first prove that \(T\) preserves distributions on
	\[
	X_\omega^{\mathrm{alg}}:=\operatorname{span}\{J_n(H_n):n\in\mathbb N\}.
	\]
	Let $x=\sum_{n=1}^{N}J_nf_n$, $f_n\in H_n$. For every \(k\in\mathbb N\), define
	\[
	g_k:=\sum_{n\in I_k\cap\{1,\dots,N\}}
	f_n	\circ\pr_n\circ \Phi_k.
	\]
	Only finitely many functions \(g_k\) are non-zero. Moreover, since
	\(\Phi_k\) preserves the product probability measure,
	\[
	\begin{aligned}
		\int_{\Omega_k}g_k\,d\mu_k=\sum_{n\in I_k\cap\{1,\dots,N\}}\int_{\Omega_k}f_n\circ\pr_n\circ\Phi_k \,d\mu_k=\sum_{n\in I_k\cap\{1,\dots,N\}}
		\int_{\Omega_n}f_n\,d\mu_n=0.
	\end{aligned}
	\]
	Thus, $g_k\in H_k$. Set $K_N:=\left\{k\in\mathbb N:I_k\cap\{1,\dots,N\}\neq\varnothing\right\}$.
	Note that $K_N=\kappa(\{1,\dots,N\})$, and hence \(K_N\) is finite. For every \(k\in K_N\), define
	\[
	Y_k:=J_kg_k, \quad 
	Z_k:=\sum_{n\in I_k\cap\{1,\dots,N\}}J_nf_n.
	\]
	
	Since \(\Phi_k\) preserves the product probability measure, the random
	variable \(g_k\) has the same distribution as
	\begin{equation}
	\label{Eq10}
	\sum_{n\in I_k\cap\{1,\dots,N\}}f_n
	\end{equation}
	on the corresponding product probability space. On the other hand,
	the variables $J_nf_n$, $n\in I_k\cap\{1,\dots,N\}$,
	depend on distinct exterior coordinates and are therefore independent.
	Since coordinate liftings do not alter distributions, the random
	variable $Z_k$
	has the same distribution as \eqref{Eq10}. Hence,
	$\dist(Y_k)=\dist(Z_k)$ for every $k\in K_N$.
	
	The family \((Y_k)_{k\in K_N}\) is independent, because its members
	depend on distinct exterior coordinates. The family
	\((Z_k)_{k\in K_N}\) is also independent, because the fibres
	\((I_k)_{k\in K_N}\) are pairwise disjoint. Therefore,
	\[
	\dist\bigl((Y_k)_{k\in K_N}\bigr)=\dist\bigl((Z_k)_{k\in K_N}\bigr).
	\]
	Applying the summation map, we obtain for every $x\in X_\omega^{\mathrm{alg}}$,
	\[ 
	\dist(Tx) = \dist\left(\sum_{k\in K_N}Y_k\right)=\dist\left(\sum_{k\in K_N}Z_k\right)=\dist(x).
	\]
	Finally, let \(x\in X_\omega\), and choose a sequence $(x^{(m)})_{m=1}^{\infty}
	\subseteq X_\omega^{\mathrm{alg}}$ such that $x^{(m)}\to x$ in \(L^p(\Omega)\). Since \(T\) is bounded, $Tx^{(m)}\to Tx$	in \(L^p(\Omega)\). Passing to a common subsequence if necessary, we may assume that $x^{(m)} \to x$ and $Tx^{(m)}\to Tx$ almost everywhere. Let $\psi:\mathbb R\to\mathbb R$ be bounded and continuous. By the dominated convergence theorem,
	\[
	\begin{aligned}
		\int_\Omega\psi(Tx)\,d\mu=\lim_{m\to\infty}
		\int_\Omega\psi(Tx^{(m)})\,d\mu=\lim_{m\to\infty}\int_\Omega\psi(x^{(m)})\,d\mu=\int_\Omega\psi(x)\,d\mu.
	\end{aligned}
	\]
	Since bounded continuous functions determine Borel probability
	measures on \(\mathbb R\), we conclude that $\dist(Tx)=\dist(x)$. Thus, \(T\) is a distributional embedding.
\end{proof}

The finite-sum representation \eqref{Eq9} in Theorem~\ref{t3.1} extends canonically to arbitrary elements of \(X_\omega\).

\begin{corollary}\label{c3.1}
	Let \(T:X_\omega\to X_\omega\) be a distributional embedding, and let $x=\sum_{n=1}^{\infty}J_nf_n\in X_\omega$
	be its canonical exterior decomposition. Then
	\[
	Tx=\sum_{k=1}^{\infty}J_k\left(\sum_{n\in I_k}f_n\circ\pr_n
	\circ\Phi_k\right),
	\]
	where each inner sum is finite and the outer series converges in
	\(L^p(\Omega)\).
\end{corollary}

\begin{proof}
	Let $x^{(N)}=\sum_{n=1}^{N}J_nf_n$. Since \(x^{(N)}\to x\) in \(L^p(\Omega)\) and \(T\) is an isometry, $Tx^{(N)}
	\to Tx$ in \(L^p(\Omega)\). Fix \(k\in\mathbb N\). By continuity of the	conditional expectation, we obtain
	\[
	\begin{aligned}
		\mathcal E_k(Tx)=\lim_{N\to\infty}\mathcal E_k(Tx^{(N)})=J_k\left(\sum_{n\in I_k}
		f_n \circ \pr_n \circ \Phi_k\right),
	\end{aligned}
	\]
	because \(I_k\) is finite. The conclusion follows from the canonical
	exterior decomposition of \(Tx\).
\end{proof}

\subsection{Some consequences} The classification becomes particularly rigid for surjective distributional embeddings as the following result shows.

\begin{corollary}\label{c3.2}
	Let $T:X_\omega\to X_\omega$ be a surjective distributional embedding. Then $\kappa(n)=n$ for every $n\in\mathbb N$. Equivalently, $I_k=\{k\}$ for every $k\in\mathbb N$. Moreover,
	\[
	T\bigl(J_k(H_k)\bigr)=J_k(H_k)\quad \text{for every }k\in\mathbb N.
	\]
	In particular, $J_k(H_k)$, $k\in\N$, are invariant subspaces of $T$.
\end{corollary}

\begin{proof}
	For every \(k\in\mathbb N\), define
	\[
	A_k:\bigoplus_{n\in I_k}H_n \longrightarrow H_k, \quad A_k\bigl((f_n)_{n\in I_k}\bigr)=\sum_{n\in I_k}f_n\circ\pr_n\circ \Phi_k.
	\]
	We claim that \(A_k\) is bijective. To prove surjectivity, let \(g\in H_k\). Since \(T\) is surjective, there exists $x=\sum_{n=1}^{\infty}J_nf_n\in X_\omega$ such that $Tx=J_kg$.
	By Corollary~\ref{c3.1},
	\[
	J_kg=\mc{E}_k(Tx)=J_k\left(\sum_{n\in I_k}f_n\circ\pr_n\circ\Phi_k\right).
	\]
	Since \(J_k\) is injective, $g=A_k\bigl((f_n)_{n\in I_k}\bigr)$.
	Thus, \(A_k\) is surjective. As a consequence, since \(H_k\neq\{0\}\), the fibre \(I_k\)
	is non-empty. To prove injectivity, assume that $A_k\bigl((f_n)_{n\in I_k}\bigr)=0$. Then
	\[
	T\left(\sum_{n\in I_k}J_nf_n\right)=0.
	\]
	Since \(T\) is an isometry, it is injective. Hence, $\sum_{n\in I_k}J_nf_n=0$. By uniqueness of the canonical exterior decomposition, $f_n=0$ for every $n\in I_k$. Therefore, \(A_k\) is injective.
	
	Since $\dim H_n=2^n-1$, the bijectivity of \(A_k\) gives
	\begin{equation}
		\label{Eq11}
	\sum_{n\in I_k}(2^n-1)=2^k-1.
	\end{equation}
	On the other hand, Theorem~\ref{t3.1} yields $\sum_{n\in I_k}n\leq k$. If \(I_k\) contained at least two elements, then
	\[
	\sum_{n\in I_k}(2^n-1)<2^{\sum_{n\in I_k}n}-1
	\leq 2^k-1,
	\]
	which would contradict the identity \eqref{Eq11}. Thus,
	\(I_k\) contains exactly one element, say $I_k=\{n\}$. The dimension count now yields
	\[
	2^n-1=\dim H_n=\dim H_k=2^k-1,
	\]
	and hence $n=k$. Therefore, $I_k=\{k\}$ for every $k\in\mathbb N$ and, in consequence, $\kappa(k)=k$. 
	
	Finally, Corollary~\ref{c2.1} gives a measure-preserving map $\varphi_k:\Omega_k\to\Omega_k$ such that $T(J_kf)=J_k(f\circ\varphi_k)$ for every $f\in H_k$.
	Since \(\Omega_k\) is a finite uniform probability space, for every \(a\in\Omega_k\) we have \[ \mu_k\bigl(\varphi_k^{-1}(\{a\})\bigr) = \mu_k(\{a\}) = 2^{-k}. \] As every atom of \(\Omega_k\) has measure \(2^{-k}\), each fibre \(\varphi_k^{-1}(\{a\})\) contains exactly one point. Hence, \(\varphi_k\) is bijective. Therefore, the pullback map $f \mapsto f\circ\varphi_k$ is a bijection of \(H_k\) onto itself, and consequently, $T\bigl(J_k(H_k)\bigr)=J_k(H_k)$.
\end{proof}

Note that the converse implication in Theorem~\ref{t3.1} also provides a direct construction of distributional embeddings from admissible packing data. Indeed, let $\kappa: \mathbb N \to \mathbb N$ be a map such that, for every \(k\in\mathbb N\), the fibre  $I_k := \kappa^{-1}(\{k\})$ is finite. Assume that $\sum_{n\in I_k}n \leq k$ for every $k\in\mathbb N$, and that, for every \(k\in\mathbb N\), there exists a measure-preserving map $\Phi_k: \Omega_k \to \prod_{n\in I_k}\Omega_n$. Then there exists a unique distributional embedding $T: X_\omega \to X_\omega$ such that 
\[
T \left( \sum_{n=1}^{N}J_nf_n \right) = \sum_{k=1}^{\infty} J_k \left( \sum_{n\in I_k\cap\{1,\dots,N\}} f_n \circ \pr_n \circ \Phi_k \right)
\] for every finite exterior sum.  

This construction admits a particularly explicit family of examples.

\begin{proposition}\label{p3.1}
	Let $\mathbb N=\bigsqcup_{j=1}^{\infty}F_j$
	be a partition of \(\mathbb N\) into finite non-empty subsets. For
	every \(j\in\mathbb N\), let \(m_j\in\mathbb N\) be such that the
	indices \(m_j\) are pairwise distinct and
	\[
	\sum_{n\in F_j}n\leq m_j.
	\]
	For every fixed \(j\in\mathbb N\), choose a family $(G_n)_{n\in F_j}$ of pairwise disjoint subsets of \(\{1,\dots,m_j\}\) such that $|G_n|=n$ for every $n\in F_j$\footnote{Such a family exists because $\sum_{n\in F_j}n\leq m_j$. Indeed, after enumerating \(F_j=\{n_1,\dots,n_r\}\), one may choose consecutive disjoint blocks of coordinates of lengths \(n_1,\dots,n_r\) inside \(\{1,\dots,m_j\}\).}. Writing $G_n=\{i_1^{(n)}<\cdots<i_n^{(n)}\}$,
	define
	\[
	\varphi_n:\Omega_{m_j}\longrightarrow\Omega_n, \quad \varphi_n(\o_1,\dots,\o_{m_j}):=(\o_{i_1^{(n)}},\dots,\o_{i_n^{(n)}}).
	\]
	Then there exists a distributional embedding $T:X_\omega\to X_\omega$ such that
	\[
	T(J_nf)=J_{m_j}(f\circ\varphi_n)
	\]
	for every $n\in F_j$ and $f\in H_n$.
\end{proposition}

\begin{proof}
	Define $\kappa(n)=m_j$ whenever $n\in F_j$.
	Then $I_{m_j}=F_j$, whereas \(I_k=\varnothing\) if $k
	\notin\{m_j:j\in\mathbb N\}$. Consequently, $\sum_{n\in I_{m_j}}n = \sum_{n\in F_j}n \leq m_j$ and $\sum_{n\in I_k}n = 0 \leq k$ for every remaining index \(k\). 
	Now, for every \(j\in\mathbb N\), define
	\[
	\Phi_{m_j}:\Omega_{m_j}\longrightarrow\prod_{n\in F_j}\Omega_n, \quad 
	\Phi_{m_j}:=	(\varphi_n)_{n\in F_j}.
	\]
	Since the subsets \(G_n\), \(n\in F_j\), are pairwise disjoint, \(\Phi_{m_j}\) is a coordinate projection and preserves the product probability measure. For every remaining index \(k\), let $\Phi_k:\Omega_k\to\prod_{n\in I_k}\Omega_n$
	be the unique map into the one-point probability space. The conclusion follows from the construction presented in the proof of Theorem~\ref{t3.1}.
\end{proof}

Let us give a particular example. For every \(j\in\mathbb N\), set $F_j	=\{2j-1,2j\}$ and $m_j=4j-1$. Then
	\[
	\sum_{n\in F_j}n=4j-1=m_j.
	\]
	Hence, Proposition~\ref{p3.1} yields a
	distributional embedding that packs the two source blocks $H_{2j-1}$, $H_{2j}$ into the single target block $H_{4j-1}$.
	More explicitly, one may define
	\[
	\varphi_{2j-1}(\o_1,\dots,\o_{4j-1})=(\o_1,\dots,\o_{2j-1}), \quad 
	\varphi_{2j}(\o_1,\dots,\o_{4j-1})=(\o_{2j},\dots,\o_{4j-1}).
	\]
	This shows that distributional self-embeddings of \(X_\omega\) need
	not act blockwise, distinct source blocks may be packed together into
	a larger target block.
	
	\section{Linear isometries of \(R_\omega^p\)}
	
	We now apply the preceding classification to the study of linear isometries of the non-centered first limit Bourgain--Rosenthal--Schechtman space \(R_\omega^p\). In our concrete realization, we have
	\[
	R_\omega^p=\mathbb R\mathbf 1\oplus R_\omega^{p,0}=\mathbb R\mathbf 1\oplus X_\omega.
	\]
	Accordingly, throughout this section we identify \(R_\omega^p\) with
	the unitalization
	\[
	\widehat X_\omega:=\mathbb R\mathbf 1\oplus X_\omega.
	\]
	The first step is to prove that, whenever \(p\neq2\), every linear isometry of \(R_\omega^p\) maps the constant function \(\mathbf 1\)	either to \(\mathbf 1\) or to \(-\mathbf 1\). For \(p\notin2\mathbb N\), Rudin's equimeasurability theorem then reduces the description of arbitrary linear isometric embeddings to the classification obtained in Theorem~\ref{t3.1}. For the larger range	\(p\neq2\), surjectivity allows us to determine the group of surjective linear isometries directly.
	
	We begin by extending the notion of exterior support to the
	unitalization. If
	\[
	f=c\mathbf 1+\sum_{n=1}^{\infty}J_nf_n \in \widehat X_\omega, \quad f_n\in H_n,
	\]
	we set $\supp_{\mathrm{ext}}(f):=\{n\in\mathbb N:f_n\neq0\}$.

	We first recall a metric characterization of disjointness in
	\(L^p\). It is a standard consequence of the scalar Clarkson inequalities and their equality cases.

	\begin{lemma}\label{l5.1}
		Assume that $1<p<\infty$, $p\neq2$. For any two real-valued functions \(f,g\in L^p(\Omega)\), the following
		statements are equivalent:
		\begin{enumerate}
			\item[{\rm (i)}] $fg=0$ almost everywhere.
			
			\item[{\rm (ii)}] $\|f+g\|_{L^p(\Omega)}^p+\|f-g\|_{L^p(\Omega)}^p=2\|f\|_{L^p(\Omega)}^p+2\|g\|_{L^p(\Omega)}^p$.
		\end{enumerate}
		Consequently, every linear isometry $U:\widehat X_\omega\to\widehat X_\omega$ preserves disjointness.
	\end{lemma}
	
	\begin{proof}
		For real numbers \(a,b\in\R\), one has
		\[
		|a+b|^p
		+
		|a-b|^p
		\begin{cases}
			>
			2|a|^p+2|b|^p,
			&
			p>2,
			\\[1ex]
			<
			2|a|^p+2|b|^p,
			&
			1<p<2,
		\end{cases}
		\]
		whenever \(ab\neq0\). Equality holds if and only if \(ab=0\).
		Integrating this pointwise identity yields the equivalence between
		{\rm (i)} and {\rm (ii)}. Let now $U:\widehat X_\omega\to\widehat X_\omega$
		be a linear isometry, and assume that \(fg=0\) almost everywhere.
		Then
		\[
		\begin{aligned}
			\|Uf+Ug\|_{L^p(\Omega)}^p+\|Uf-Ug\|_{L^p(\Omega)}^p
			&=\|f+g\|_{L^p(\Omega)}^p+
			\|f-g\|_{L^p(\Omega)}^p=2\|f\|_{L^p(\Omega)}^p+2\|g\|_{L^p(\Omega)}^p \\
			&=2\|Uf\|_{L^p(\Omega)}^p+2\|Ug\|_{L^p(\Omega)}^p.
		\end{aligned}
		\]
		Hence, $(Uf)(Ug)=0$ almost everywhere.
	\end{proof}
	
	The additive structure of \(\widehat X_\omega\) imposes a strong
	restriction on disjoint elements.
	
	\begin{lemma}\label{l5.2}
		Let $f,g\in\widehat X_\omega$ be non-zero functions such that $fg=0$ almost everywhere.
		Then
		\[
		\left|\supp_{\mathrm{ext}}(f)\cup\supp_{\mathrm{ext}}(g)\right|\leq2.
		\]
	\end{lemma}
	
	\begin{proof}
		Write
		\[
		f
		=
		a\mathbf 1
		+
		\sum_{n=1}^{\infty}J_nf_n,
		\qquad
		g
		=
		b\mathbf 1
		+
		\sum_{n=1}^{\infty}J_ng_n,
		\]
		where $a,b\in\mathbb R$ and $f_n,g_n\in H_n$. Fix two distinct indices $r,s\in\mathbb N$. We may write
		\[
		f=F+J_rf_r+J_sf_s, \qquad g=G+J_rg_r+J_sg_s,\]
		where \(F\) and \(G\) are measurable with respect to $\sigma(\pi_n:n\notin\{r,s\})$.
		Indeed, after subtracting the \(r\)-th and \(s\)-th exterior
		components, the remaining partial sums are measurable with respect to
		this sigma-algebra and converge in \(L^p(\Omega)\). The corresponding
		subspace of \(L^p(\Omega)\) is closed. Since $fg=0$ almost everywhere, Fubini's theorem and the finiteness of $\Omega_r\times\Omega_s$ yield a set
		\[
		E_{r,s}\subseteq\prod_{n\notin\{r,s\}}\Omega_n
		\]
		of full measure such that, for every $\xi\in E_{r,s}$, we have
		\[
		\bigl(F(\xi)+f_r(\alpha)+f_s(\gamma)\bigr)\bigl(G(\xi)+g_r(\alpha)+g_s(\gamma)\bigr)=0
		\]
		for all $\alpha\in\Omega_r$ and $\gamma\in\Omega_s$. Fix $\xi\in E_{r,s}$, set $A:=F(\xi)$, $C:=G(\xi)$, and define
		\[
		Q(\alpha,\gamma):=\bigl(A+f_r(\alpha)+f_s(\gamma)\bigr)\bigl(C+g_r(\alpha)+g_s(\gamma)\bigr).
		\]
		Then $Q(\alpha,\gamma)=0$ for every $\alpha\in\Omega_r$ and $\gamma\in\Omega_s$. Let $\alpha,\beta\in\Omega_r$ and $\gamma,\delta\in\Omega_s$. Taking the difference
		\[
		Q(\alpha,\gamma) -Q(\alpha,\delta)-Q(\beta,\gamma)+Q(\beta,\delta),
		\]
		we obtain zero. Expanding \(Q\), all terms depending only on the
		\(r\)-th coordinate or only on the \(s\)-th coordinate cancel. The
		only remaining terms are the mixed ones. Hence,
		\[
			\bigl(f_r(\alpha)-f_r(\beta)\bigr)\bigl(g_s(\gamma)-g_s(\delta)\bigr)+\bigl(g_r(\alpha)-g_r(\beta)\bigr)\bigl(f_s(\gamma)-f_s(\delta)\bigr)=0.
		\]
		For each \(n\in\mathbb N\), define
		\[
		V_n:=\operatorname{span}\left\{\left(f_n(\alpha)-f_n(\beta),
		g_n(\alpha)-g_n(\beta)\right):\alpha,\beta\in\Omega_n\right\}\subseteq\mathbb R^2.
		\]
		Consider the non-degenerate symmetric bilinear form $B:\mathbb R^2\times\mathbb R^2\to\mathbb R$ defined by
		\[
		B\bigl((u,v),(u',v')\bigr):=uv'+vu'.
		\]
		The preceding identity shows that $B(V_r,V_s)=\{0\}$ whenever $r\neq s$. Now, assume, towards a contradiction, that there exist three distinct indices $n_1,n_2,n_3\in\mathbb N$
		such that $V_{n_j}\neq\{0\}$, $j=1,2,3$. Since the subspaces \(V_{n_1},V_{n_2},V_{n_3}\) are pairwise
		orthogonal with respect to the non-degenerate bilinear form \(B\) on \(\mathbb R^2\), they must all be contained in the same isotropic\footnote{Recall that a subspace \(L\) is called isotropic if $B(v,w)=0$ for every $v,w\in L$. For a one-dimensional subspace \(L=\mathbb R v\), this is equivalent to \(B(v,v)=0\).} line. Let us show this.
		Choose a non-zero vector $v_1\in V_{n_1}$. Since $B(V_{n_1},V_{n_j})=\{0\}$, $j=2,3$, we have
		\[
		V_{n_2},V_{n_3}\subseteq v_1^\perp.
		\]
		As \(B\) is non-degenerate on \(\mathbb R^2\), the orthogonal subspace \(v_1^\perp\) is one-dimensional. Since both \(V_{n_2}\) and \(V_{n_3}\) are non-zero, it follows that
		\[
		V_{n_2}=V_{n_3}=v_1^\perp.
		\]
		Moreover, $B(V_{n_2},V_{n_3})=\{0\}$, so the line $L:=v_1^\perp$ is isotropic. Finally, $V_{n_1}\subseteq L^\perp=L$. Hence, $V_{n_1},V_{n_2},V_{n_3}\subseteq L$, where \(L\) is an isotropic line.
		
		The only isotropic lines for \(B\) are $\mathbb R(1,0)$ and $\mathbb R(0,1)$. Interchanging \(f\) and \(g\) if necessary, we may assume that $V_{n_j}\subseteq\mathbb R(1,0)$, $j=1,2,3$. Hence, \(g_{n_j}\) is constant for every \(j\). Since $g_{n_j}\in H_{n_j}$, we obtain $g_{n_j}=0$, $j=1,2,3$.
		
		In particular, $g_{n_1}=0$. Hence, \(g\) does not depend on the exterior coordinate indexed by
		\(n_1\). More precisely, writing
		\[
		\Omega=\Omega_{n_1}\times\Omega', \quad \Omega':=\prod_{n\neq n_1}\Omega_n,
		\]
		there exists a measurable function $\widetilde g:\Omega'\to\R$ and a set $E_0\subseteq\Omega'$ of full measure such that $g(\alpha,\xi)=\widetilde g(\xi)$ for every $\alpha\in\Omega_{n_1}$
		and every $\xi\in E_0$. Here we use the finiteness of \(\Omega_{n_1}\). Since \(g\neq0\), the set
		\[
		E_1:=\left\{\xi\in\Omega':\widetilde g(\xi)\neq0\right\}
		\]
		has positive measure. On the other hand, since $fg=0$ almost everywhere, again Fubini's theorem and the finiteness of \(\Omega_{n_1}\) yield a set $E_2\subseteq\Omega'$ of full measure such that 
		\begin{equation}
			\label{Eq19}
			f(\alpha,\xi)g(\alpha,\xi)= 0
		\end{equation}
		for every $\alpha\in\Omega_{n_1}$ and $\xi\in E_2$. Write $f=F+J_{n_1}f_{n_1}$, where \(F\) does not depend on the \(n_1\)-st exterior coordinate. After choosing a representative of \(F\), there exists a set $E_3
		\subseteq\Omega'$ of full measure such that
		\[
		f(\alpha,\xi)=F(\xi)+f_{n_1}(\alpha)
		\]
		for every $\alpha\in\Omega_{n_1}$ and $\xi\in E_3$. Choose $\xi\in E_0 \cap E_1\cap E_2\cap E_3$.
		Then $g(\alpha,\xi)=\widetilde g(\xi)\neq 0$ for all $\alpha\in\Omega_{n_1}$. Consequently, by \eqref{Eq19}, $f(\alpha,\xi)=0$ for every $\alpha\in\Omega_{n_1}$. Therefore,
		\[
		0=F(\xi)+f_{n_1}(\alpha), \quad \forall \a\in\O_{n_1}.
		\]
		Hence, \(f_{n_1}\) is constant. Since $f_{n_1} \in H_{n_1}$,
		we obtain $f_{n_1}=0$. This contradicts $V_{n_1}\neq\{0\}$. Therefore, there exist at most two indices \(n\) such that $V_n\neq\{0\}$.
		
		Finally, $V_n=\{0\}$ if and only if both \(f_n\) and \(g_n\) are constant. Since $f_n,g_n\in H_n$, this is equivalent to $f_n=g_n=0$. Hence, $\left|\supp_{\mathrm{ext}}(f)\cup\supp_{\mathrm{ext}}(g)\right|\leq2$.
	\end{proof}
	
	We next show that a non-constant element of \(\widehat X_\omega\)
	cannot be decomposed into arbitrarily many non-zero disjoint
	fragments.
	
	\begin{lemma}\label{lem:bounded-disjoint-fragmentation}
		Let $h\in\widehat X_\omega$ be non-constant. Then there exists a finite constant $M(h)<\infty$
		such that every decomposition 
		\[h=y_1+\cdots+y_q\]
		into non-zero elements $y_1,\dots,y_q\in\widehat X_\omega$
		satisfying $y_iy_j=0$ almost everywhere for every \(i\neq j\), necessarily satisfies
		\[
		q\leq M(h).
		\]
	\end{lemma}
	
	\begin{proof}
		The conclusion is trivial when $q=1$. Assume that $q\geq2$. By Lemma~\ref{l5.2}, for every \(i\neq j\),
		\begin{equation}	
			\label{Eq17}
		\left|\supp_{\mathrm{ext}}(y_i)\cup\supp_{\mathrm{ext}}(y_j)\right|\leq2.
		\end{equation}
		Set $E:=\bigcup_{j=1}^q\supp_{\mathrm{ext}}(y_j)$.
		We claim that $|E|\leq2$. Indeed, suppose first that $\left|\supp_{\mathrm{ext}}(y_i)\right|=2$ for some \(i\).
		Since \(\supp_{\mathrm{ext}}(y_i)\) already contains two indices, it follows from \eqref{Eq17} that
		\[
		\supp_{\mathrm{ext}}(y_j)\subseteq \supp_{\mathrm{ext}}(y_i).
		\]
		Hence, $|E|\leq 2$. Assume now that $\left|\supp_{\mathrm{ext}}(y_i)\right|\leq 1$ for every $i$.
		No \(y_i\) can be a non-zero constant function. Indeed, since
		\(q\geq2\), there exists \(j\neq i\), and the condition $y_i y_j=0$ almost everywhere would force \(y_j=0\), a contradiction. Therefore, $\supp_{\mathrm{ext}}(y_i)$ is a singleton for every \(i\). Suppose that two fragments \(y_i\) and \(y_j\) depended on distinct exterior coordinates. Then they would be measurable with respect to
		independent sigma-algebras. Since both fragments are non-zero, we have $\mu(\{y_i\neq0\})>0$ and $\mu(\{y_j\neq0\})>0$. By independence,
		\[
		\mu\bigl(\{y_i\neq0\}\cap\{y_j\neq0\}\bigr)=\mu(\{y_i\neq0\})\mu(\{y_j\neq0\})>0.
		\]
		This contradicts $y_i y_j=0$ almost everywhere. Consequently, all the singleton exterior supports coincide. Thus, $|E|\leq2$.
		
		Since $h=\sum_{j=1}^q y_j$, we have $\supp_{\mathrm{ext}}(h) \subseteq E$. We distinguish three cases.
		Assume first that $\left|\supp_{\mathrm{ext}}(h)\right|\geq 3$.
		Then no decomposition with \(q\geq2\) exists. In this case, we may take $M(h)=1$.
		Assume next that $\supp_{\mathrm{ext}}(h)=\{r,s\}$ for distinct indices \(r,s\). In this case, necessarily, $E=\{r,s\}$. Then, every fragment \(y_j\) is measurable with respect to $\sigma(\pi_r,\pi_s)$. The finite probability space $\Omega_r\times\Omega_s$ has \(2^{r+s}\) atoms. Since the fragments are non-zero and pairwise disjoint, we obtain $q\leq 2^{r+s}$.
		
		Finally, assume that $\supp_{\mathrm{ext}}(h)=\{r\}$. If $E=\{r\}$, then every \(y_j\) depends only on the \(r\)-th exterior coordinate. Since \(\Omega_r\) has \(2^r\) atoms, we obtain $q\leq2^r$. It remains to consider the case $E=\{r,s\}$ for some \(s\neq r\). Every fragment can then be written as
		\[
		y_j=c_j\mathbf 1+J_ru_j+J_sv_j,
		\]
		where $c_j\in\mathbb R$,  $u_j\in H_r$ and $v_j\in H_s$. We claim that, for every $\eta\in\Omega_s$, the section $y_j(\,\cdot\,,\eta)$ is non-zero for every \(j\). Assume, towards a contradiction, that $y_j(\,\cdot\,,\eta_0)\equiv 0$ for some \(j\) and some \(\eta_0\in\Omega_s\). Then $c_j+u_j(\alpha)+v_j(\eta_0)=0$ for every $\alpha\in\Omega_r$. Hence, \(u_j\) is constant. Since $u_j\in H_r$, we obtain $u_j=0$. Thus, \(y_j\) depends only on the \(s\)-th exterior coordinate. Since \(y_j\neq0\), there exists $\eta_1\in\Omega_s$ such that $y_j(\eta_1)\neq 0$. The pairwise disjointness of the fragments implies that $y_i(\alpha,\eta_1)=0$ for every $\alpha\in\Omega_r$ whenever \(i\neq j\). Therefore,
		\[
		h(\alpha)=\sum_{i=1}^q y_i(\alpha,\eta_1)=y_j(\eta_1)
		\quad \text{for every }\alpha\in\Omega_r.
		\]
		This contradicts the fact that \(h\) is non-constant and depends only
		on the \(r\)-th exterior coordinate. Thus, for every \(\eta\in\Omega_s\), the functions $y_1(\,\cdot\,,\eta),\dots,y_q(\,\cdot\,,\eta)$ are non-zero and pairwise disjoint on \(\Omega_r\). Since $|\Omega_r|=2^r$, we conclude that $q\leq2^r$. The proof is complete.
	\end{proof}

	We may now prove that the image of the constant function under a
	linear isometry is necessarily constant.
	
	\begin{proposition}\label{p5.1}
		Assume that $p\neq2$ and let $U:\widehat X_\omega\to\widehat X_\omega$
		be a linear isometry. Then there exists $\varepsilon\in\{-1,1\}$ such that
		\[
		U\mathbf 1=\varepsilon\mathbf 1.
		\]
	\end{proposition}
	
	\begin{proof}
		Fix \(n\in\mathbb N\). For every atom \(a\in\Omega_n\), the indicator $J_n\mathbf 1_{\{a\}}$
		belongs to \(\widehat X_\omega\). The family $J_n\mathbf 1_{\{a\}}$, $a\in\Omega_n$, consists of non-zero pairwise disjoint functions and satisfies
		\[
		\sum_{a\in\Omega_n}J_n\mathbf 1_{\{a\}}=\mathbf 1.
		\]
		By Lemma~\ref{l5.1}, the functions $U(J_n\mathbf 1_{\{a\}})$, $a\in\Omega_n$, are pairwise disjoint. Since \(U\) is an isometry, they are non-zero. Moreover,
		\[
		\sum_{a\in\Omega_n}U\left(J_n\mathbf 1_{\{a\}}\right)=U\mathbf 1.
		\]
		Thus, \(U\mathbf 1\) can be decomposed into \(2^n\) non-zero pairwise
		disjoint fragments for every \(n\in\mathbb N\).
		
		If \(U\mathbf 1\) were non-constant,
		Lemma~\ref{lem:bounded-disjoint-fragmentation} would provide a finite
		upper bound for the number of fragments in such a decomposition. This
		is impossible because $2^n\to\infty$ as $n\to\infty$.
		Hence, $U\mathbf 1=c\mathbf 1$ for some \(c\in\mathbb R\). Since \(U\) is an isometry,
		\[
		1=\|\mathbf 1\|_{L^p(\Omega)}=\|U\mathbf 1\|_{L^p(\Omega)}=|c|.
		\]
		Therefore, $c\in\{-1,1\}$.
	\end{proof}
	
	We now use the real version of Rudin's equimeasurability theorem. We
	recall that, if $p\notin2\mathbb N$, and \(f,g\) are real-valued \(L^p\)-functions on probability spaces
	such that
	\[
	\|\mathbf 1+tf\|_p=\|\mathbf 1+tg\|_p
	\]
	for every $t\in\mathbb R$, then $\dist(f)=\dist(g)$ (see the original reference \cite{R76} or \cite{KK01}).
	
	\begin{theorem}\label{t5.1}
		Assume that $1<p<\infty$, $p\notin2\mathbb N$ and let $U:\widehat X_\omega\to\widehat X_\omega$ be a linear isometry. Then there exist $\varepsilon\in\{-1,1\}$ and a distributional embedding $T:X_\omega\to X_\omega$
		such that
		\[
		U=\varepsilon\widehat T.
		\]
		Consequently, every linear isometry of \(\widehat X_\omega\), up to a global change of sign, is described by the finite packing data of Theorem~\ref{t3.1}.
	\end{theorem}
	
	\begin{proof}
		By Proposition~\ref{p5.1}, there exists $\varepsilon\in\{-1,1\}$ such that $U\mathbf 1=\varepsilon\mathbf 1$. Define $V:=\varepsilon U$. Then $V\mathbf 1=\mathbf 1$. Now, let $x\in X_\omega$. Then, for every \(t\in\mathbb R\), we have
		\[
			\|\mathbf 1+tVx\|_{L^p(\Omega)}=\left\|V\left(\mathbf 1+tx\right)\right\|_{L^p(\Omega)}=
			\|\mathbf 1+tx\|_{L^p(\Omega)}.
		\]
		By the Rudin's equimeasurability theorem, $\dist(Vx)=\dist(x)$. In particular,
		\[
		\int_\Omega Vx\,d\mu=\int_\Omega x\,d\mu=0.
		\]
		Hence, $V(X_\omega)\subseteq X_\omega$. Define $T:=V|_{X_\omega}$. Then $T:X_\omega\to X_\omega$
		is a distributional embedding. Moreover, for every $c\in\mathbb R$ and $x\in X_\omega$, we have
		\[
		V(c\mathbf 1+x)=c\mathbf 1+Tx=\widehat T(c\mathbf 1+x).
		\]
		Therefore, $V=\widehat T$. Since $V=\varepsilon U$,
		we conclude that $U=\varepsilon\widehat T$.
	\end{proof}

\subsection{Surjective linear isometries}

Theorem \ref{t5.1} classifies arbitrary linear isometric embeddings of \(R_\omega^p\) when $p\notin2\mathbb N$.
For surjective linear isometries, however, this restriction can be relaxed. Indeed, surjectivity allows us to recover the exterior blocks directly from the preservation of disjointness and the Boolean rigidity theorem.

\begin{proposition}\label{p5.2}
	Assume that $p\neq 2$ and let $V:\widehat X_\omega\to\widehat X_\omega$ be a surjective linear isometry satisfying
	$V\mathbf 1=\mathbf 1$. Then, for every \(n\in\mathbb N\), there exists a bijection $\varphi_n:\Omega_n\to\Omega_n$ such that
	\[
	V(J_nf)=J_n(f\circ\varphi_n)
	\]
	for every $f\in H_n$. Consequently, $V(X_\omega)=X_\omega$,
	and
	\[
	V\left(c\mathbf 1+\sum_{n=1}^{\infty}J_nf_n\right)
	=c\mathbf 1+\sum_{n=1}^{\infty}J_n(f_n\circ\varphi_n)
	\]
	for every $c\in\mathbb R$ and every canonical exterior decomposition $\sum_{n=1}^{\infty}J_nf_n\in X_\omega$.
\end{proposition}

\begin{proof}
	Although the proof parallels some arguments used in the classification of distributional embeddings, it cannot be deduced directly from Theorem~\ref{t3.1}. Indeed, at this stage we do not know that the restriction of \(V\) to \(X_\omega\) preserves distributions. We shall therefore recover the block structure of \(V\) directly from the preservation of disjointness and the Boolean rigidity established in Theorem~\ref{t2.1}. The surjectivity assumption will only be used afterwards to show that each exterior block is mapped onto itself.
	
	Fix $n\in\mathbb N$ and for every atom $a\in \Omega_n$, set $e_a:=J_n\mathbf 1_{\{a\}}$.
	Then $e_a(\mathbf 1-e_a)=0$ almost everywhere. By Lemma~\ref{l5.1}, \(V\) preserves disjointness and since $V\mathbf 1=\mathbf 1$, we obtain
	\[
	(Ve_a)(\mathbf 1-Ve_a)=0
	\]
	almost everywhere. Therefore, \(Ve_a\) is an indicator function.
	More precisely, there exists a measurable set $A_a\subseteq\Omega$ such that $Ve_a=\mathbf 1_{A_a}$.
	Moreover, since \(V\) is an isometry,
	\[
	\mu(A_a)=\|Ve_a\|_{L^p(\Omega)}^p=\|e_a\|_{L^p(\Omega)}^p=2^{-n}.
	\]
	In particular, \(A_a\) is non-trivial. By Theorem~\ref{t2.1}, there exist a unique index $\kappa_a\in
	\mathbb N$ and a non-trivial subset $B_a\subseteq\Omega_{\kappa_a}$ such that
	\[
	\mathbf 1_{A_a}=J_{\kappa_a}\mathbf 1_{B_a}
	\]
	almost everywhere. As in the proof of Proposition~\ref{p2.3}, the
	index \(\kappa_a\) is independent of \(a\). Indeed, if two such
	indicators depended on distinct exterior coordinates, the
	corresponding events would be independent and of positive measure,
	which is incompatible with their disjointness. Hence, there exists an index $\kappa(n) \in \mathbb N$ such that $\kappa_a = \kappa(n)$ for every $a \in \Omega_n$. Since $\sum_{a\in\Omega_n}e_a=\mathbf 1$,
	linearity of \(V\) gives
	\[
	\sum_{a\in\Omega_n}J_{\kappa(n)}\mathbf 1_{B_a}=\mathbf 1.
	\]
	Thus, the family $(B_a)_{a\in\Omega_n}$ is a partition of $\Omega_{\kappa(n)}$. Define $\varphi_n:\Omega_{\kappa(n)}\to\Omega_n$ by $\varphi_n(\omega):=a$ whenever $\omega\in B_a$.
	For every $a\in\Omega_n$, we have
	\[
	\mu_{\kappa(n)}\bigl(\varphi_n^{-1}(\{a\})\bigr)=\mu_{\kappa(n)}(B_a)=2^{-n}=\mu_n(\{a\}).
	\]
	Hence, \(\varphi_n\) preserves the uniform probability measure. By
	linearity,
	\[
	V(J_nf)=J_{\kappa(n)}(f\circ\varphi_n)
	\]
	for every $f\in L^p(\Omega_n)$. In particular, this identity holds for every $f\in H_n$. Since \(\varphi_n\) preserves the probability measure, it also follows that $f\circ\varphi_n\in H_{\kappa(n)}$
	whenever $f\in H_n$. Therefore,
	\[
	V\bigl(J_n(H_n)\bigr)\subseteq J_{\kappa(n)}\bigl(H_{\kappa(n)}\bigr).
	\]
	Notice also that $\kappa(n)\geq n$, because every fibre of \(\varphi_n\) has measure \(2^{-n}\), whereas
	each atom of \(\Omega_{\kappa(n)}\) has measure \(2^{-\kappa(n)}\). Since
	\[
	X_\omega
	=\overline{\operatorname{span}\{J_n(H_n):n\in\mathbb N\}}^{\,L^p(\Omega)},
	\]
	the preceding inclusion and the continuity of \(V\) imply that $V(X_\omega)\subseteq X_\omega$.
	
	We now use surjectivity to prove the reverse inclusion. Let $y\in X_\omega$.
	There exist $c\in\mathbb R$ and $x\in X_\omega$ such that
	\[
	y=V(c\mathbf 1+x)= c\mathbf 1+Vx.
	\]
	Since both \(y\) and \(Vx\) have mean zero, we obtain $c=0$.
	Hence, $y=Vx$, and therefore $V(X_\omega)=X_\omega$.
	
	For every $k\in\mathbb N$, set $I_k:=\kappa^{-1}(\{k\})$.
	The family of subspaces $V\bigl(J_n(H_n)\bigr)$, $n\in I_k$ is algebraically direct. Indeed, if
	\[
	\sum_{n\in F}V(J_nf_n)=0
	\]
	for some finite set $F\subseteq I_k$ and some functions $f_n\in H_n$, then the injectivity of \(V\) gives
	\[
	\sum_{n\in F}J_nf_n=0.
	\]
	By uniqueness of the exterior decomposition, $f_n=0$ for every $n\in F$.
	Since all these non-zero subspaces are contained in the finite-dimensional space $J_k(H_k)$, it follows that $I_k$ is finite. 
	
	Now, we claim that
	\[
	J_k(H_k)
	=
	\bigoplus_{n\in I_k}
	V\bigl(J_n(H_n)\bigr).
	\]
	Only the inclusion from left to right requires proof. Let $y\in J_k(H_k)$. Since $V(X_\omega)=X_\omega$, there exists $x=\sum_{n=1}^{\infty}J_nf_n\in X_\omega$
	such that $y=Vx$. By continuity of \(V\), 
	\[
	Vx = \lim_{N\to\infty} \sum_{n=1}^N V(J_nf_n) 
	\] 
	in \(L^p(\Omega)\). Moreover, 
	\[
	\mathcal E_k\bigl(V(J_nf_n)\bigr) = 
	\begin{cases} 
		V(J_nf_n), & \kappa(n)=k, \\[1ex] 
		0, & \kappa(n)\neq k, 
	\end{cases} 
    \]
    because $V(J_nf_n) \in J_{\kappa(n)}\bigl(H_{\kappa(n)}\bigr)$. Since \(I_k\) is finite, the continuity of \(\mathcal E_k\) yields 
    \[ 
    y = \mathcal E_k(y) = \mathcal E_k(Vx) = \sum_{n\in I_k} V(J_nf_n). 
    \] 
    This proves the claim. Taking dimensions yields
	\[
	2^k-1=\sum_{n\in I_k}(2^n-1).
	\]
	Since $\kappa(n)\geq n$ every $n\in I_k$ satisfies $n\leq k$. We now show that $I_k=\{k\}$.
	If $k\notin I_k$, then
	\[
	\sum_{n\in I_k}(2^n-1)\leq \sum_{n=1}^{k-1}(2^n-1)=2^k-k-1<2^k-1,
	\]
	which is impossible. On the other hand, if $k\in I_k$
	and \(I_k\) contains another index, then
	\[
	\sum_{n\in I_k}(2^n-1)>2^k-1,
	\]
	which is also impossible. Therefore, $I_k=\{k\}$.
	In particular, $\kappa(k)=k$ for every $k\in\mathbb N$. We have thus obtained a measure-preserving self-map $\varphi_k:\Omega_k\to\Omega_k$ for every $k\in\mathbb N$. Since \(\Omega_k\) is a finite uniform probability space, every measure-preserving self-map is bijective. Hence, $\varphi_k$ is bijective.
	Finally, the formula for arbitrary canonical exterior decompositions follows from the density of finite exterior sums and the continuity of \(V\).
\end{proof}

\begin{corollary}\label{c5.1}
	Assume that $p\neq 2$ and let $U:R_\omega^p\to R_\omega^p$ be a surjective linear isometry. Then there exist
	$\varepsilon\in\{-1,1\}$ and, for every $k\in\mathbb N$, a bijection $\varphi_k:\Omega_k\to\Omega_k$
	such that
	\[
	U\left(c\mathbf 1+\sum_{k=1}^{\infty}J_kf_k\right)
	=\varepsilon c\mathbf 1+\varepsilon\sum_{k=1}^{\infty}J_k(f_k\circ\varphi_k)
	\]
	for every $c\in\mathbb R$
	and every canonical exterior decomposition $\sum_{k=1}^{\infty}J_kf_k\in X_\omega$.
\end{corollary}

\begin{proof}
	By Proposition~\ref{p5.1}, there exists $\varepsilon\in\{-1,1\}$ such that $U\mathbf 1=\varepsilon\mathbf 1$. Define $V:=\varepsilon U$. Then \(V\) is a surjective linear isometry satisfying $V\mathbf 1=\mathbf 1$.
	The conclusion follows from Proposition~\ref{p5.2}.
\end{proof}

Note that the preceding corollary gives an explicit description of the
group of surjective linear isometries of \(\widehat X_\omega\). For a
finite set \(A\), we denote by $\operatorname{Sym}(A)$ the group of all bijections from \(A\) onto itself. Thus, an element of $\prod_{k=1}^{\infty}\operatorname{Sym}(\Omega_k)$ is a sequence $(\varphi_k)_{k=1}^{\infty}$, $\varphi_k:\Omega_k\to\Omega_k$, where each \(\varphi_k\) is a permutation of the atoms of
\(\Omega_k\).

Each such sequence, together with a global sign $\varepsilon\in\{-1,1\}$,
determines a surjective linear isometry by
\[
c\mathbf 1+\sum_{k=1}^{\infty}J_kf_k\longmapsto\varepsilon c\mathbf 1+\varepsilon
\sum_{k=1}^{\infty} J_k(f_k\circ\varphi_k).
\]
Conversely, every surjective linear isometry has this form.

Strictly speaking, the correspondence with the product group is
contravariant, since pullbacks compose in the reverse order. The inversion map identifies each symmetric group with its opposite group, so the contravariance of pullbacks causes no difficulty. Hence, for every $1<p<\infty$, $p\neq2$, the group of surjective linear isometries of \(\widehat X_\omega\) is naturally isomorphic to
\[
\{-1,1\}\times\prod_{k=1}^{\infty}\operatorname{Sym}(\Omega_k).
\]
Finally, since $|\Omega_k|=2^k$, the group $\operatorname{Sym}(\Omega_k)$
is isomorphic to the symmetric group on \(2^k\) symbols, which we
denote by $\mathfrak S_{2^k}$.
Therefore, the group of surjective linear isometries is isomorphic to
\[
\{-1,1\}\times\prod_{k=1}^{\infty}\mathfrak S_{2^k}.
\]
The exclusion of the Hilbertian case \(p=2\) is necessary. Indeed, the exterior blocks \(J_n(H_n)\) are mutually orthogonal in \(L^2(\Omega)\), and
	\[
	R_\omega^2
	=
	\mathbb R\mathbf 1
	\oplus_2
	\left(
	\bigoplus_{n=1}^{\infty}
	J_n(H_n)
	\right)_2
	\]
	is a separable infinite-dimensional Hilbert space. Hence, $R_\omega^2\cong\ell^2$ isometrically, and its group of surjective linear isometries is the full orthogonal group. In particular, an isometry need not preserve
	the constant functions or the exterior blocks.

\section{Proper internal compressions do not exist}

We conclude the paper by studying whether \(X_\omega\) can contain a distributional copy of itself localized on a proper measurable subset of the underlying probability space. This question is naturally formulated in terms of compressions. A \(\theta\)-compression preserves the original distribution with weight \(\theta\) and vanishes with the
remaining probability \(1-\theta\). Thus, it may be heuristically viewed as a distributional copy
localized on a region of measure \(\theta\).

Compressions arise naturally as elementary components in the construction of distributional embeddings between higher
Bourgain--Rosenthal--Schechtman spaces. In the present setting, however,
we show that no proper internal compression exists. We use the terminology of Konstantos and Motakis \cite{KM25}.

\begin{definition}\label{d4.1}
	Let \(0<\theta\leq 1\). A linear operator
	\(S \colon X_\omega \to X_\omega\) is called a
	\emph{\(\theta\)-compression} if
	\[
	\dist(Sx)=\theta\dist(x)+(1-\theta)\delta_0\quad\text{for every } x \in X_\omega.
	\]
\end{definition}
Equivalently, for every bounded Borel function
\(\psi \colon \mathbb{R} \to \mathbb{R}\), we have
\[
\int_\Omega \psi(Sx) \, d\mu= \theta \int_\Omega \psi(x) \, d\mu + (1-\theta)\psi(0).
\]
Applying this identity to $\psi_N(t)=\min\{|t|^p,N\}$ and letting \(N\to\infty\), the monotone convergence theorem yields $\lVert Sx \rVert_{L^p(\Omega)}^p=\theta\lVert x \rVert_{L^p(\Omega)}^p$, and hence $\lVert Sx \rVert_{L^p(\Omega)}=\theta^{1/p}\lVert x \rVert_{L^p(\Omega)}$. Thus, every non-zero compression is injective.

Although Definition~\ref{d4.1} is stated in terms of scalar
distributions, linearity implies a corresponding statement for
finite-dimensional joint distributions.

\begin{lemma}\label{l4.1}
	Let \(S \colon X_\omega \to X_\omega\) be a
	\(\theta\)-compression, and let \(x_1,\dots,x_m \in X_\omega\). Then
	\[
	\dist\bigl(Sx_1,\dots,Sx_m\bigr)=\theta\dist\bigl(x_1,\dots,x_m\bigr)+(1-\theta)\delta_{(0,\dots,0)}.
	\]
\end{lemma}

\begin{proof}
	For every vector \(t=(t_1,\dots,t_m) \in \mathbb{R}^m\), linearity of
	\(S\) gives
	\[
	\sum_{j=1}^m t_jSx_j=S\left(\sum_{j=1}^m t_jx_j\right).
	\]
	By Definition~\ref{d4.1},
	\[
	\dist\left(\sum_{j=1}^m t_jSx_j\right)=\theta
	\dist\left(\sum_{j=1}^m t_jx_j\right)+(1-\theta)\delta_0.
	\]
	The right-hand side is precisely the distribution of the scalar
	projection in the direction \(t\) of the probability measure
	\[
	\theta\dist\bigl(x_1,\dots,x_m\bigr)+(1-\theta)\delta_{(0,\dots,0)}.
	\]
	Since all one-dimensional projections agree, the Cramér--Wold theorem
	yields the conclusion.
\end{proof}

We choose a Rademacher variable in the first exterior block. Let
\(r \in H_1=L_p^0(\{0,1\})\) be defined by $r(0)=1$, $r(1)=-1$.
We identify \(r\) with its canonical copy \(J_1r \in X_\omega\) and
keep the same notation. Thus, $\dist(r)=\frac{1}{2}\delta_{-1}+\frac{1}{2}\delta_1$.

\begin{lemma}\label{l4.2}
	Let \(S \colon X_\omega \to X_\omega\) be a
	\(\theta\)-compression, and define
	\[
	A:=\{\omega \in \Omega : \lvert Sr(\omega)\rvert=1\}.
	\]
	Then \(\mu(A)=\theta\), and $Sx=0$ almost everywhere on $A^c$
	for every \(x \in X_\omega\).
\end{lemma}

\begin{proof}
	Since \(S\) is a \(\theta\)-compression,
	\[
	\dist(Sr)=\frac{\theta}{2}\delta_{-1}+(1-\theta)\delta_0+\frac{\theta}{2}\delta_1.
	\]
	Therefore, \(\mu(A)=\theta\), and \(A^c=\{Sr=0\}\) modulo null sets. Fix \(x \in X_\omega\). By
	Lemma~\ref{l4.1},
	\[
	\dist(Sx,Sr)=\theta\dist(x,r)+(1-\theta)\delta_{(0,0)}.
	\]
	The source variable \(r\) never vanishes, hence, \(\dist(x,r)\) assigns zero mass to
	\(\mathbb{R}\times\{0\}\). Consequently, \(\dist(Sx,Sr)\) assigns no mass to $(\mathbb{R}\setminus\{0\})\times\{0\}$. Equivalently,
	\[
	\mu\bigl(\{Sx\neq 0\}\cap\{Sr=0\}\bigr)=0.
	\]
	Since \(A^c=\{Sr=0\}\) modulo null sets, we conclude that \(Sx=0\) almost everywhere on \(A^c\).
\end{proof}

The common support obtained in Lemma~\ref{l4.2} will depend on only finitely many exterior coordinates. We now prove that a proper event of this type cannot support an infinite-dimensional subspace of \(X_\omega\).

\begin{lemma}\label{l4.3}
	Let \(F \subseteq \mathbb{N}\) be finite, and let $A\in\sigma(\pi_n:n \in F)$
	satisfy \(\mu(A)<1\). Then
	\[
	\mc{X}_{A}:=\left\{x \in X_\omega : x=0 \text{ almost everywhere on } A^c \right\} \subseteq \bigoplus_{n \in F} J_n(H_n).
	\]
	In particular, the subspace $\mc{X}_{A}$
	is finite-dimensional.
\end{lemma}

\begin{proof}
	Let $x=\sum_{n=1}^{\infty} J_nx_n\in X_\omega$ be supported in \(A\), and fix \(m \notin F\). We prove that \(x_m=0\). Since \(x=0\) almost everywhere on \(A^c\), we have \(\mathbf{1}_{A^c}x=0\). Taking conditional expectation with respect to \(\Sigma_m=\sigma(\pi_m)\), we obtain $\mathbb{E}\left[\mathbf{1}_{A^c}x\mid \Sigma_m\right]=0$. We decompose
	\[
	x=J_mx_m+\sum_{n\neq m} J_nx_n.
	\]
	Because \(A^c\) depends only on the exterior coordinates indexed by \(F\), and \(m \notin F\), the event \(A^c\) is independent of \(\Sigma_m\). Therefore,
	\[
	\mathbb{E} \left[\mathbf{1}_{A^c}J_mx_m\mid\Sigma_m \right]=\mu(A^c)J_mx_m.
	\]
	
	For the remaining terms, we distinguish two cases. If
	\(n \notin F\cup\{m\}\), then \(\mathbf{1}_{A^c}J_nx_n\) is independent of \(\Sigma_m\), and
	\[
	\int_\Omega \mathbf{1}_{A^c}J_nx_n\, d\mu=\mu(A^c)\int_\Omega J_nx_n \, d\mu=0.
	\]
	Hence, $\mathbb{E}\left[\mathbf{1}_{A^c}J_nx_n\mid\Sigma_m\right]=0$.
	If \(n \in F\), then \(\mathbf{1}_{A^c}J_nx_n\) is also independent
	of \(\Sigma_m\). Consequently,
	\[
	\mathbb{E}\left[\mathbf{1}_{A^c}J_nx_n\mid\Sigma_m\right]
	=\left(\int_\Omega\mathbf{1}_{A^c}J_nx_n\, d\mu\right)\mathbf{1}.
	\]
	
	Multiplication by \(\mathbf{1}_{A^c}\) and conditional expectation
	with respect to \(\Sigma_m\) are bounded operators on \(L^p(\Omega)\).
	We may therefore apply them to the convergent exterior decomposition
	of \(x\) and pass to the limit. Since \(F\) is finite, the preceding
	identities give
	\[
	0=\mu(A^c)J_mx_m+c_m\mathbf{1}
	\]
	for some scalar \(c_m \in \mathbb{R}\). Integrating both sides over \(\Omega\), and using the fact that \(x_m \in H_m\), we obtain
	\[
	0=\mu(A^c)\int_\Omega J_mx_m \, d\mu+c_m=c_m.
	\]
	Thus, \(c_m=0\). Since \(\mu(A^c)>0\), it follows that \(J_mx_m=0\), and hence \(x_m=0\). This holds for every \(m \notin F\). Therefore,
	\[
	x\in\bigoplus_{n \in F} J_n(H_n).
	\]
	This completes the proof.
\end{proof}

We can now prove the main result of this section.

\begin{theorem}
	Let \(S \colon X_\omega \to X_\omega\) be a \(\theta\)-compression, where \(0<\theta\leq 1\). Then \(\theta=1\). Equivalently, every non-zero compression of \(X_\omega\) into itself is a distributional embedding.
\end{theorem}

\begin{proof}
	Assume, towards a contradiction, that \(0<\theta<1\). Let
	\(r \in X_\omega\) be the Rademacher variable chosen above, and define $A:=\{\lvert Sr\rvert=1\}$.
	By Lemma~\ref{l4.2}, we have \(\mu(A)=\theta<1\), and $Sx=0$ almost everywhere on $A^c$ for every \(x \in X_\omega\). The random variable \(Sr\) has distribution
	\[
	\frac{\theta}{2}\delta_{-1}+(1-\theta)\delta_0+\frac{\theta}{2}\delta_1.
	\]
	Hence, \(Sr\) has finite essential range. By
	Proposition~\ref{p2.2}, its
	exterior support $F:=\supp_{\mathrm{ext}}(Sr)$
	is finite. Since \(Sr\) depends only on the exterior coordinates indexed by \(F\), the event \(A=\{\lvert Sr\rvert=1\}\) belongs to \(\sigma(\pi_n:n \in F)\). By
	Lemma~\ref{l4.3},
	\[
	\operatorname{Ran}(S)\subseteq\bigoplus_{n \in F} J_n(H_n).
	\]
	The right-hand side is finite-dimensional. On the other hand, \(S\) is injective because
	\[
	\lVert Sx \rVert_{L^p(\Omega)}=\theta^{1/p}\lVert x \rVert_{L^p(\Omega)}
	\]
	for every \(x \in X_\omega\). Thus, \(S\) defines an injective linear
	map from the infinite-dimensional space \(X_\omega\) into a
	finite-dimensional space. This is impossible. Therefore, \(\theta=1\).
\end{proof}

\section*{Acknowledgements}

The author is grateful to P. Motakis and K. Konstantos for their interest in this work, for reading a preliminary version of the manuscript, and for several valuable comments and suggestions. Their observations were particularly helpful in clarifying the connection between distributional embeddings and linear isometries of \(R_\omega^p\).

\end{document}